\documentclass[12pt,a4paper,oneside]{amsart}
\usepackage{amsmath,amssymb,amsthm,hyperref}
\usepackage{enumerate}
\usepackage{amsfonts}
\usepackage{amsthm}
\usepackage{amsmath}
\usepackage{amssymb}
\usepackage{t1enc}
\usepackage{graphicx}
\usepackage{mathrsfs}
\usepackage{mathptmx}
\usepackage{lipsum}
\usepackage{tikz}
\usepackage{version} 
\usepackage{color}
\usepackage{enumerate}
\numberwithin{equation}{section}
\usepackage[bf]{caption}     
\usepackage{geometry}        
\usepackage{bm}  
\usepackage{buczosier}
\usepackage{mathtools}
\usepackage[utf8]{inputenc}
\usepackage[normalem]{ulem}

\newtheorem{theorem}{Theorem}[section]

\newtheorem{lemma}[theorem]{Lemma}

\newtheorem{corollary}[theorem]{Corollary}
\theoremstyle{definition}
\newtheorem{remark}[theorem]{Remark}
\newtheorem{definition}[theorem]{Definition}

\newcommand{\conv}{\mathrm{conv}}

\definecolor{aquam}{rgb}{0.5,1.0,1.0}
\definecolor{bbrown}{rgb}{0.75,0.38,0.15}
\definecolor{Cyan}{rgb}{0,0.6,0.6}
\definecolor{Darkblue}{rgb}{0,0,1}
\definecolor{Dodgerblue2}{rgb}{0,0.5,1}
\definecolor{Green}{rgb}{0,0.6,0.06}
\definecolor{Kahki}{rgb}{1,1,0.5}
\definecolor{Magenta}{rgb}{0.7,0,0.7}
\definecolor{bMagenta}{rgb}{1,.6,1}
\definecolor{Orange}{rgb}{0.8,0.3,0}
\definecolor{dOrchid}{rgb}{0.7,0.2,0.4}
\definecolor{Orchid}{rgb}{1,0.5,1}
\definecolor{Purple}{rgb}{0.65,0.07,0.85}
\definecolor{Royalblue}{rgb}{0.6,0.85,0.87}
\definecolor{Tan}{rgb}{0.54,0.42,0.23}
\definecolor{bTan}{rgb}{0.94,0.82,0.63}
\definecolor{zoltan}{rgb}{0,0.1,0.3}
\definecolor{Turquoise}{rgb}{0,0.85,0.87}
\definecolor{Yellow}{rgb}{1,1,0}
\definecolor{darkamber}{rgb}{0.4,0.19,0.28}
\definecolor{bYellow}{rgb}{1,1,0.6}
\definecolor{bRed}{rgb}{1,0.7,0.7}
\definecolor{boxcolb}{rgb}{0.87,0.77,0.75}
\definecolor{boxcol}{rgb}{0.6,0.85,0.87}
\definecolor{boxcolgreen}{rgb}{0.64,0.93,0.79}
\definecolor{boxcolaa}{rgb}{.75,.99,.70}
\definecolor{boxcolbb}{rgb}{0.39,0.50,0.56}
\definecolor{boxcolcc}{rgb}{1,0.81,0.65}
\definecolor{yy}{rgb}{0.43,0.21,.18}
\definecolor{gA}{gray}{0.5}
\definecolor{gB}{gray}{0.8}
\definecolor{gC}{gray}{0.9}

\usepackage{xcolor}

\hypersetup{
  colorlinks   = true, 
  urlcolor     = blue, 
  linkcolor    = blue, 
  citecolor   = red 
}

\begin{document}

\title{ Generic H\"older level sets and fractal conductivity }
\author{Zolt\'an Buczolich$^*$}
\address{Department of Analysis, ELTE E\"otv\"os Lor\'and\\
University, P\'azm\'any P\'eter S\'et\'any 1/c, 1117 Budapest, Hungary}
\email{zoltan.buczolich@ttk.elte.hu}
\urladdr{http://buczo.web.elte.hu, ORCID Id: 0000-0001-5481-8797}

\author{Bal\'azs Maga$^\text{\textdagger}$}
\address{Department of Analysis, ELTE E\"otv\"os Lor\'and\\
University, P\'azm\'any P\'eter S\'et\'any 1/c, 1117 Budapest, Hungary}
\email{mbalazs0701@gmail.com}
\urladdr{  http://magab.web.elte.hu/}

\author{ G\'asp\'ar V\'ertesy$^\text{\textdaggerdbl}$}
 \address{ Alfr\'ed R\'enyi Institute of Mathematics, Re\'altanoda street 13-15, 1053 Budapest, Hungary}
\email{vertesy.gaspar@gmail.com}

\thanks{\scriptsize $^*$
 The project leading to this application has received funding from the European Research Council (ERC) under the European Union’s Horizon 2020 research and innovation programme (grant agreement No. 741420).
This author was also supported by the Hungarian National Research, Development and Innovation Office--NKFIH, Grant 124003 and  at the time of completion of this paper was holding a visiting researcher position 
at the Rényi Institute.
}
\thanks{\scriptsize $^\text{\textdagger}$ This author was supported by the \'UNKP-21-3 New National Excellence of the Hungarian Ministry of Human Capacities, and by the Hungarian National Research, Development and Innovation Office-NKFIH, Grant 124749.}
\thanks{\scriptsize $^\text{\textdaggerdbl}$  This author was supported by the \'UNKP-20-3 New National Excellence Program of the Ministry for Innovation and Technology from the source of the National Research, Development and Innovation Fund, and by the Hungarian National Research, Development and Innovation Office–NKFIH, Grant 124749. 
}

\date{\today}


\begin{abstract}
Hausdorff dimensions of level sets of generic continuous functions defined on fractals 
can give information about the ``thickness/narrow cross-sections'' of a ‘‘network’’ corresponding to a fractal set, $F$.
This lead to the definition of the topological Hausdorff dimension of fractals.
In this paper we continue our study of the level sets of generic $1$-H\"older-$\aaa$ functions. While in a previous paper we gave the initial definitions and established some properties of these generic level sets, in this paper we provide
numerical estimates in the case of the Sierpi\'nski triangle. These calculations give better insight and illustrate why can one think of these generic $1$-H\"older-$\aaa$
level sets as something measuring ``thickness/narrow cross-sections/conductivity''
of  a fractal ‘‘network’’.\\
We also give an example for the phenomenon which we call  phase transition for $D_{*}(\aaa, F)$.  This roughly means that for a certain lower range { of $\aaa$s } only the geometry
of $F$ determines $D_{*}(\aaa, F)$ while for larger values the H\"older exponent, $\aaa$ also matters.
\end{abstract}

\maketitle

\setcounter{tocdepth}{3}

\tableofcontents


\section{Introduction}
  
In \cite{BBEtoph} the concept of topological Hausdorff dimension was introduced.
It has turned out that this concept was related to some ``conductivity/narrowest cross-section'' properties of fractal sets and hence was mentioned and used  in Physics papers, see for example 
\cite{Balankintoph}, \cite{Balankinfracspace}, \cite{Balankintransport}, \cite{Balankintoph2}, and  \cite{Balankinfluid}.
In \cite{MaZha}  the authors studied topological Hausdorff dimension of fractal squares.

The starting point of the research leading to the definition of topological Hausdorff dimension was a purely theoretic question concerning 
the Hausdorff dimension of the level sets of generic continuous functions defined on fractals, apart from  \cite{BBEtoph} see also \cite{BBElevel}. 
(Some people prefer to use the term typical in the Baire category sense instead 
of generic.) 
It has turned out that these generic level sets ``can find'' the narrowest cross-sections'' of the fractals on which they are defined. 
For example the Sierpi\'nski triangle has many
zero dimensional ``cross-sections'' and the Sierpi\'nski carpet has many $\log 2/\log 3$ dimensional cross sections. 
These are the values  of the Hausdorff dimensions of the level sets for almost every levels $y$ in the range of the generic continuous
functions defined on them.
 On Figure \ref{crossed} the red lines crossing the Sierpi\'nski carpet are intersecting the fractal in $\log 2/\log 3$ dimensional sets.
In case of the Sierpi\'nski triangle near the red lines one can obtain curves intersecting the fractal in finitely many points. 
\begin{figure}[ht]
\includegraphics[width=0.95\textwidth]{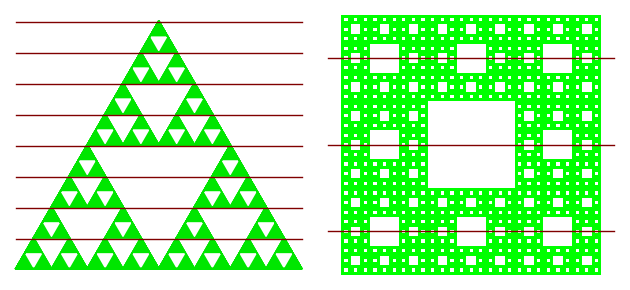}
\caption{Fourth  level approximations of the Sierpi\'nski triangle and carpet with ``narrow cross-sections''\label{crossed}}
\end{figure}

The level sets of generic continuous functions are ``infinitely compressible'' which  informally speaking  means that we can squeeze almost every level in the range of
the continuous function to places where the fractal domain is the ``thinnest'', yielding that the dimension of these level sets are small.
Consider the Sierpi\'nski triangle for example: even though one cannot squeeze almost every level set onto a curve which intersects the triangle in a countable set of points,
for  a  generic continuous function almost every level set stays ``near'' such curves, yielding that they have zero dimension.
 It is a natural question what happens if the level sets/regions  of  our functions are not ``infinitely compressible'' and hence due to thickness of the level regions 
we cannot use for almost every levels the parts of our fractal domains where they are the ``thinnest''. The simplest way to impose a bound on compressibility is considering H\"older functions instead of arbitrary continuous functions. More specifically and explicitly, it is straightforward to observe the level sets of generic 1-H\"older-$\aaa$ functions. 
In  our  recent paper \cite{sier} we started to deal with this question.

|n Section \ref{*secprel} we give and recall some definitions 
and theorems from \cite{sier}
and prove some preliminary results.
 Out of these definitions we mention here that $D_{*}(\aaa, F)$ denotes the  essential supremum of the Hausdorff dimensions of the level sets of a generic
$1$-H\"older-$\aaa$ function defined on $F$.
In \cite{sier} we showed that if $\alpha\in(0,1)$, then for connected self{-}{}similar sets, like the Sierpi\'nski triangle
$D_{*}(\aaa, F)$ equals  the Hausdorff dimension 
of almost every level-set in the range of  a generic 1-H\"older-$\aaa$ function, that is, it is not necessary to consider the essential supremum.

In Sections \ref{*secsier} and \ref{*secsierb} we consider $1$-H\"older-$\aaa$ functions defined on the Sierpi\'nski triangle, $\DDD$. 
In Theorem \ref{allexpthm} we obtain a lower estimate for $D_{*}(\aaa,\DDD)$ which equals in this case the Hausdorff dimension 
of almost every level-set of {\it any 1-H\"older-$\aaa$ function } defined on $\DDD$.

\begin{figure}[ht]
\includegraphics[width=0.5\textwidth]{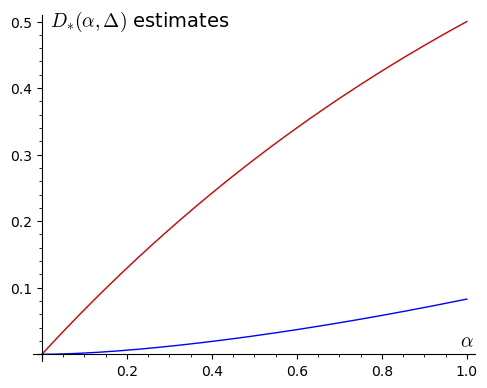}
\caption{Lower and upper estimates of $D_{*}(\aaa,\DDD)$ \label{figsierestimates}}
\end{figure}

In Theorem \ref{*thsierfb} for the Hausdorff dimension 
of almost every level-set of  {\it a generic 1-H\"older-$\aaa$ function } we also calculate an upper estimate. 
Level sets of 1-H\"older-$\aaa$ functions can get quite complicated.  In  some very special cases when either the function is linear, or 
there is a direction such that it is constant on lines pointing in this direction
we need to consider intersections of these lines with our fractal.
Intersection of fractals with lines is a classical topic (see for example \cite{Marstrplane}), which even in the case of the Sierpi\'nski triangle, or carpet is  
still subject of more recent research as well, see
\cite{BBFASK}
\cite{LXZhao}.

In Section \ref{*secssc} we show that if our fractal $F$ is a self-similar set 
satisfying the strong separation condition then 
 the Hausdorff dimension 
of almost every level-set of  a generic 1-H\"older-$\aaa$ function  is constant
zero for all $\aaa\in (0,1)$, that is the introduction of generic 1-H\"older-$\aaa$ functions
is not giving any new information compared to the case of continuous functions.

In Section \ref{*secphase} we discuss and illustrate a phenomenon which we call 
phase transition. 
In Theorem \ref{phasetransfatcantor} we
 give an example of a fractal $F$ for which   
  the Hausdorff dimension 
of almost every level-set of  a generic 1-H\"older-$\aaa$ function 
 for small values $\aaa$ 
equals the Hausdorff dimension 
of almost every level-set of  a generic continuous function defined on $F$.
 This  means, at a heuristic level, that for such fractals the level sets of 
generic 1-H\"older-$\aaa$ functions are as flexible/compressible as those
of a continuous function. 
 On the other hand,  for larger values of $\aaa$ we have $D_{*}(\aaa, F)>0$, that is after a critical value of $\aaa$ these level sets 
are not as flexible/compressible as those of a continuous function and we experience some ``traffic'' jams as we try to push across the fractal the level sets
of generic 1-H\"older-$\aaa$ functions.
The fractal $F$ discussed in this section will be the Cartesian product of a  fat Cantor set with  itself,  hence it will be of zero topological dimension. Note that due to Section \ref{*secssc}, such a construction requires fat Cantor sets. Indeed, a self-similar Cantor set cannot have the above properties.
In \cite{sier} we proved that $D_{*}(\aaa, F)$ is monotone increasing
in $\aaa$ for any compact set $F$. 
It is a natural question whether this function is continuous.
For the fractal $F$ of Theorem \ref{phasetransfatcantor} there is not only a 
phase transition at $\aaa_{\phi}=\frac{1}{2}$, but also a jump discontinuity
of $D_{*}(\aaa, F)$. { In \cite{sier} we give an example of a fractal for which $D_{*}(\aaa, F)$ has a jump discontinuity at $0$ and for that fractal there is no phase transition.}
{ The behaviour exhibited by these fractals is in contrast with the behaviour of the Sierpi\'nski triangle $\Delta$: 
as demonstrated by the bounds given in Section \ref{*secsier} and \ref{*secsierb}, $D_{*}(\aaa, \Delta)$ is continuous and increasing in 0, ruling out 
the possibility of phase transition.}



\section{Preliminaries}\label{*secprel}

In this section first we recall some definitions and results from 
\cite{sier}.

The distance of $x,y\in \R^{p}$ is denoted by $|x-y|$. 
If $A\sse \R^{p}$ then the diameter of $A$ is denoted by $|A|=\sup\{ |x-y|: x,y\in A \}.$
The open ball of radius $r$ centered at $x$ is denoted by $B(x,r)$. 

Assume that $F\subseteq \mathbb{R}^p$ for some $p>0$. 
In what follows, $F$ will be some fractal set, usually we suppose that it is compact. 

We say that a function $f: F \to \mathbb{R}$ is $c$-H\"older-$\alpha$ for $c>0$ and $0<\alpha\leq 1$ if $|f(x)-f(y)|\leq c|x-y|^\alpha$ { for all $x,y\in F$.} 
The space of such functions will be denoted by $C^{\aaa}_{c}(F)$, or if $F$ is fixed then by $C^{\aaa}_{c}$. 
The space of H\"older-$\aaa$ functions will be denoted by $C^{\aaa}$, that is $C^{\aaa}= \bigcup_{c>0}  C^{\aaa}_{c}.$
We say that $f$ is $c^-$-H\"older-$\alpha$ if there exists $c'<c$ such that
$f$ is $c'$-H\"older-$\alpha$. The set of such functions is denoted by $C^{\aaa}_{c^{-}}$, that is $C^{\aaa}_{c^{-}}=\bigcup_{c'<c}C^{\aaa}_{c'}$.

 In order to introduce topology on the set $C^{\aaa}(F)$, we think of it as a subset of continuous functions 
equipped with the supremum norm $||f||_{\oo}=\sup_{x\in F}|f(x)|.$  To obtain a closed subset of $C^{\aaa}(F)$, 
we  take  $1$-H\"older-$\alpha$ functions, $C^{\aaa}_{1}(F)$, 
and use the metric coming from the supremum norm.

For $\rrr>0$ and $f\in C(F)$ we denote by $B(f,\rrr)$ the open ball of radius $\rrr$
centered at $f$, the ball taken in the supremum norm. 
If $f\in C_{1}^{\aaa}(F)$ then
$B(f,\rrr)\cap C_{1}^{\aaa}(F)$ will denote the corresponding open ball in the subspace $C_{1}^{\aaa}(F).$

Since similarities are not changing the geometry  of  a fractal set to avoid some 
unnecessary technical difficulties we suppose that we work with fractal sets $F$
of diameter at most one (unless otherwise stated). 
This way
\begin{equation}\label{*caaaineq}
C^{\aaa}_{1}(F)\sse C^{\aaa'}_{1}(F) \text{  if $\aaa>\aaa'$.}
\end{equation}

Suppose $A\sse \R^{p}$. 
Given $\ddd>0$ the sets $U_{j}$, $\ddd\text{-cover }A$
if $|U_{j}|<\ddd$ for all $j$ and $A\sse \bigcup_{j} U_{j}$.

The $s$-dimensional Hausdorff measure (see its definition for example in \cite{[Fa1]}) is denoted by $\cah^{s}$. 
Recall that the Hausdorff dimension of $A\sse \R^{p}$
is given by
 \begin{equation}\label{*defdimh}
  \dim_H A=\inf\{ s : \cah^{s}(A)=0\}=
\end{equation}
  $$\inf\{s:\ex \mathbf{C}_s>0,\ \ax\ddd>0,\  \ex \{ U_{j} \} \text{ a } \ddd\text{-cover of }A  \text{ s.t. } \sum_{j}|U_{j}|^{s}<\mathbf{C}_s \}.$$

We will use the Mass Distribution Principle, see for example \cite{[Fa1]}, Chapter 4.


\begin{theorem}\label{thMDP}
Let $\mu$ be a mass distribution (a finite, non-zero  Borel  measure) on $F\subset \R^p$. 
Suppose that for some 
$s\geq 0$ there are numbers $c>0$ and $\ddd>0$ such that $\ds \mmm(U)\leq c|U|^{s}$
for all sets $U$ with $|U|\leq \ddd$. 
Then $\cah^{s}(F)\geq \mu(F)/c$ and $s\leq \dim F.$
\end{theorem}

Let $D^f(r,F)=D^f(r)=\dim_H(f^{-1}(r))$ for any function $f: F\to \mathbb{R}$, that is $D^f(r)$ denotes the Hausdorff dimension of the function $f$ at level $r$.

We are interested in the Hausdorff dimension of the level sets for sufficiently large sets of levels in the sense of Lebesgue. Hence we put 
\begin{displaymath}
D_{*}^f(F)=\sup\{d: \lambda\{r : D^f(r,F)\geq{d}\}>0\},
\end{displaymath}
where $\lambda$ denotes the one-dimensional Lebesgue measure.

The definition of $D_*^f(F)$  depends on $f$. 
In case we want a definition depending only on the fractal $F$ we can first take 
\begin{displaymath}
\underline{D}_{*}(\alpha,F)=\inf\{D_{*}^f: f:F\to\mathbb{R} \text{ is  locally  non-constant and $1$-H\"older-}\alpha\},
\end{displaymath}
where the locally non-constant property is understood as $f$ is non-constant on $U\cap F$ where $U$ is any neighborhood of any accumulation point of $F$.

As we are only concerned with nonnegative numbers, by convention the infimum of the empty set is $0$. 
The value $\underline{D}_{*}(\alpha,F)$ concerns those functions for which 
``most'' level sets are smallest possible.

We denote by $\mg_{1,\aaa}(F)$, or by simply $\mg_{1,\aaa}$ the  set  of dense $G_{\ddd}$
sets in $C_{1}^{\aaa}(F)$.

We put
\begin{equation}\label{*defDcsaF}
D_{*}(\aaa,F)=\sup_{\cag\in \mg_{1,\aaa}}\inf\{ D_{*}^{f}(F):f\in \cag \}.
\end{equation}

We recall Theorem 5.2 of \cite{sier}, which states that one can think of $D_{*}(\aaa,F)$ in a substantially simplified form, essentially due to the supremum being a maximum, for which the infimum is taken over a singleton. Precisely:
 
 \begin{theorem} \label{*thmgenex}
If   $0< \aaa\leq 1$  and $F\subset\R^p$ is compact, then there is a dense $G_\delta$ subset $\cag$ of $C_1^\alpha(F)$ such that for every $f\in\cag$ we have $D_*^f(F) = D_*(\alpha,F)$.
\end{theorem}

 As we mentioned in the introduction, another result from \cite{sier} 
implies that if  $0<\aaa<1 $  and $F$ is a connected self{-}{}similar set then
 one can think of $D_{*}(\aaa,F)$ as 
the Hausdorff dimension of almost every level set in the range of  a  generic 
$C_1^\alpha(F)$
function. 

To include generic continuous functions 
in our notation we set  $$D_{*}(0, F)=\max\{0,\dim_{tH} F-1\},$$ where $\dim_{tH} F$
is the topological Hausdorff dimension of $F$.
For the definition see \cite{BBEtoph}.
By results in \cite{BBEtoph}
we have  that if $f$ is  a  generic continuous function on $F$, then 
$$D_{*}(0, F)=D_{*}^f(F).$$

For brevity, often we will omit $F$ from our notation.

Recall from \cite{sier} the following trivial upper bound for $D_*(\alpha, F) $.

\begin{theorem} \label{thm:trivial_upper_bound}
For any bounded measurable set $F\subseteq \mathbb{R}^{p}$, we have $$D_*(\alpha, F) \leq \max(0, \overline{\dim}_B (F) - 1).$$
{ (As usual, $ \overline{\dim}_B (F)$ denotes the upper box dimension of $F$.)}
 \end{theorem}

In \cite{sier} we state and prove several approximation results. We will use the next two lemmas later in this paper.

\begin{lemma} \label{lipschitzapprox}
Assume that $F$ is compact and $c>0$ is fixed. 
Then the Lipschitz $c$-H\"older-$\alpha$ functions defined on $F$ form a dense subset of the $c$-H\"older-$\alpha$ functions.
 (We say that $f$ is Lipschitz $c$-H\"older-$\alpha$, if $f$ is Lipschitz and $f\in C^{\aaa}_c(F)$.)
\end{lemma}

\begin{lemma} \label{piecewise_affine_approx}
Assume that $F$ is compact, $0<\alpha<1$, and $0<c$ are fixed. 
Then the locally non-constant piecewise affine $c^{-}$-H\"older-$\alpha$ functions defined on $F$ form a dense subset of the $c$-H\"older-$\alpha$ functions.
\end{lemma}

\section{Lower estimate for arbitrary functions on the Sierpi\'nski triangle}\label{*secsier}

In this section as an example we consider the Sierpi\'nski triangle, $ F=\DDD\subseteq\mathbb{R}^2$.  It is a connected self-similar set. Hence by a result of \cite{sier} $D_{*}(\aaa, \DDD)$ equals  the Hausdorff dimension 
of almost every level-set of  a generic 1-H\"older-$\aaa$ function.

  Some people prefer to work with different versions of the 
 Sierpi\'nski triangle. We work with the one which is obtained by starting with an equilateral triangle of side length one. 
Hence it satisfies our earlier assumptions about the fractals considered since its diameter 
 equals one.
 Its topological Hausdorff dimension equals one and this implies that
for  a  generic continuous function every level set is zero-dimensional, see  
\cite{BBEtoph}. The level sets of continuous functions are very flexible, and very ``compressible'', hence  during the proof of this theorem
one can capitalize on  the fact that
the Sierpi\'nski triangle is very ``thin'' near the vertices of the small triangles
appearing during its construction. 
 As H\"older-$\aaa$ functions do not have this flexibility, one can expect that their level sets {  generically} exhibit a different behaviour. 
In Theorem \ref{allexpthm} we obtain a lower bound 
for the Hausdorff dimension of
 almost every level set in the range of an arbitrary 1-H\"older-$\aaa$ function  defined on 
$\DDD$. 
This lower estimate is positive for $\aaa>0$.
In Theorem \ref{*thsierfb} we give an upper estimate for the Hausdorff dimension of almost every level set   of  the generic 1-H\"older-$\aaa$  function defined on 
$\DDD$, that is we estimate $D_{*}(\alpha, \Delta)$ from above. 
Both the lower and upper estimates tend to $0$ as $\aaa\to 0+0$.
Of course, it would be interesting to 
determine  the exact value of the function $D_{*}(\alpha, \Delta)$, but this seems to be quite difficult.

By its definition the  Sierpi\'nski triangle  is expressible as $\Delta=\bigcap_{n=0}^{\infty}\Delta_n$ where $\Delta_n$ is the union of the triangles appearing at the $n$th step of the construction. 
The set of triangles on the $n$th level is $\tau_n$. 
For $T\in\tau_n$ we denote by $V(T)$ the set of its vertices. 
Moreover, let ${\mathbf V}_{n}$ be the set of the points which are vertices of some $T\in\tau_{n}$, and their union is ${\mathbf V}  = \bigcup_{n=0}^{\infty} \mathbf V_{n} $. 
We are interested in the Hausdorff dimension of the level sets of a 1-H\"older-$\alpha$ function $f:\Delta\to\mathbb{R}$ for $0<\alpha\leq 1$.

\begin{figure}[ht]
\includegraphics[width=0.55\textwidth]{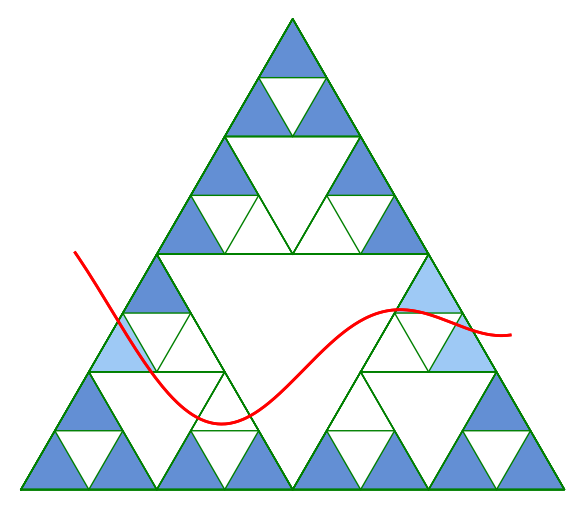}
\caption{Sierpi\'nski triangle and a crossing level set, shaded blue triangles are used during the definition of $\Delta^3$, the lighter shaded triangles correspond to the first level approximation $G^{3}_{1}(r)$ of the red level set at level $r$  \label{figlevel}}
\end{figure}

Suppose $l\in\N$.
It will be useful for us to define the self-similar set $\Delta^l\subseteq \Delta$ as well. 
 It  is induced by the similarities which map $\Delta_0$ to any triangle $T\in \tau_l$ on the boundary of $\Delta_0$. For example, $\Delta^1 = \Delta$, while the case $l=3$ is shown by Figure \ref{figlevel} where the shaded triangles on the sides of $\DDD_{0}$ are used in the definition of $\DDD^{3}$. (The lighter shaded triangles will have importance later.)

 One can easily check  that the number of triangles used in the construction of $\Delta^l$ is $3(2^l-1)$, and the $n$th level of $\Delta^l$ consists of certain triangles of $\tau_{nl}$. 
Let us denote the family of these triangles by $\tau_n^l$, the union of their vertices for fixed $n$ by ${\mathbf V}_n^l$, and the union of vertices for all the triangles in some $\tau_n^l$ by ${\mathbf V}^l=\bigcup_{n=1}^{\infty} {\mathbf V}_n^l$.  
It is clear that a 1-H\"older-$\alpha$ function $f:\Delta\to\mathbb{R}$ restricted to $\Delta^l$ is still a 1-H\"older-$\alpha$ function.

Suppose that $f:\Delta^l\to\mathbb{R}$ is a 1-H\"older-$\alpha$ function
and $r\not \in f({\mathbf V})$.
 We can define the $n$th approximation of $f^{-1}(r)$ denoted by $G_{n}^l(r)$ for any $n$ and $r\in f(\Delta^l)$ as the union of some triangles in $\tau_{n}^l$. 
More explicitly, $T\in\tau_{n}^l$ is taken into $G_{n}^l(r)$ if and only if $T$ has vertices $v$ and $v'$ such that $f(v)< r < f(v')$, that is, $r$ is in the 
interior of the convex hull $\conv (f(V(T)))$. 
The idea is that in this case $f^{-1}(r)$ necessarily intersects $T$.
On Figure \ref{figlevel} the level set corresponding to $f^{-1}(r)$ is the intersection of
 the red curve with $\DDD$. 
The set $G_{1}^3(r)$ consists of the light shaded triangles.
Now it is easy to check that,  using the notation $\conv$ for the convex hull, 
\begin{equation}
 \label{approxlevelunion} \conv (f(V(T)))\subseteq\bigcup_{T'\in\tau_{n+1}^l,\ T'\subseteq T}\conv (f(V(T'))),
\end{equation}
hence if $G_{n}^l(r)$ contains a triangle $T\in\tau_{n}^l$ then $G_{n+1}^l(r)$ contains a triangle $T'\in\tau_{n+1}^l$ such that $T'\subseteq T$. 
We introduce the following terminology: we say that $T'\in\tau_{n+k}^l$ is the $l,r$-descendant of $T\subseteq G_{n}^l(r)$ if there exists a sequence $T_0=T\supseteq T_1\supseteq...\supseteq T_k=T'$ of triangles such that $T_i\in\tau_{n+i}^l$ and $T_i\subseteq G_{n+i}^l(r)$ for $i=0,1,...,k$. 
We denote the set of $l,r$-descendants of $T$ by $\mathcal{D}_r^l(T)$.

Observe the obvious property that for any $T$ we can label the vertices in $V(T)$ such that $f(v_0)\leq f(v_2) \leq f(v_1)$. 
We refer to $v_0,v_1$ as the extreme vertices of $T$. 
Since we supposed that $r\in \inte\conv (f(V(T)))$ we have $f(v_{1})>f(v_{0})$.
If $f(v_{0})=f(v_{2})$ then we call only one vertex $v_{0}$  as  an extreme
vertex, the other vertex denoted by $v_{2}$ will not regarded to be an extreme vertex.
We proceed analogously if $f(v_{1})=f(v_{2})$.
  
We define the conductivity $\kappa_{n}^l(T)=\kappa_{n}^l(T,f)$ of any triangle $T\in\tau_{n}^l$ inductively (as $f$ is fixed during most of our arguments, it will be omitted from the notation unless it might cause ambiguity).
 If $n=0$, we define $\kappa_0^l(T)=1$. 
On the other hand, if $n\geq{1}$, there is a unique triangle $T'\in\tau_{n-1}^l$ such that $T\subseteq T'$. 
Now if $T$ is one of the two triangles at an extreme vertex of $T'$, then let $\kappa_{n}^l(T)=\kappa_{n-1}^l(T')$ (in this case we say that $T$ is an extreme triangle of $T'$), while in any other case we let $\kappa_{n}^l(T)=\frac{1}{2}\kappa_{n-1}^l(T')$. 
 The following lemma 
can be thought of  as the weak  conservation of conductivity:


\begin{lemma}
\label{approxlevelconductivity}
Assume that $T\in G_{n}^l(r)$ and $k\geq 1$. 
Then we have
\begin{displaymath}
\sum_{T'\in \tau_{n+k}^l\cap\mathcal{D}_{r}^{l}(T)}\kappa_{n+k}^l(T')\geq{\kappa_{n}^l(T)}.
\end{displaymath}
\end{lemma}

\begin{proof}
By induction, it suffices to work with $k=1$. 
Consider the vertices $\nu_{0}$ and $\nu_{1}$ on which $f$ is minimal and maximal respectively in $V(T)$.
  Since $r\not\in f({\mathbf V})$ and $f(\nu_{0})<r<f(\nu_{1})$
there are at least two edges of $T$ containing points of $f^{-1}(r)$.
One of them is the one connecting  $\nu_0$ and $\nu_{1}$.

If there is a $T'\in \tau_{n+1}^l$ which contains all the intersection
points of the edges of $T$ and $f^{-1}(r)$ then it should contain $\nu_{0}$, or $\nu_{1}$. Hence, it is an extreme triangle of $T$ and the conductivity of $T'$
equals that of $T$.

Otherwise we have at least two triangles of  $G_{n+1}^l(r)$ which are in $T$
and the sum of their conductivity is at least the conductivity of $T$.
\end{proof}

Now we have enough tools to turn our attention to the Hausdorff dimension of the level sets of a 1-H\"older-$\alpha$ function $f:\Delta\to\mathbb{R}$.


\begin{theorem}
\label{allexpthm}
Assume that $f:\Delta\to\mathbb{R}$ is a 1-H\"older-$\alpha$ function for some $0<\alpha\leq 1$. 
Then for Lebesgue almost every $r\in f(\Delta)$ we have
\begin{equation}\label{*expe}
\dim_H(f^{-1}(r))\geq \frac{\frac{\alpha}{2}}{1+\frac{1+\log{\frac{3}{\alpha}}}{\log 2}+\frac{2}{\alpha}}> 0.
\end{equation}
\end{theorem}

\begin{proof}[Proof of Theorem \ref{allexpthm}]

Since $\DDD$ is compact and connected $f(\Delta)$ is a closed interval.
Moreover as ${\mathbf V}$ is a countable and dense subset of $\DDD$ the set 
$f({\mathbf V})$ is countable and dense in $f(\DDD)$.
Suppose that $r\in \intt( f(\DDD))\sm f({\mathbf V})$.
Then we can find $T\in \bigcup_n \tau_n$ 
 such that $r\in \inte \conv (f(V(T)))$. 
Due to self-similarity properties, we can assume $T=\Delta_0$.

Restrict $f$ to some $\Delta^l$. 
The number $l$ will be fixed later, it is useful to think of it as something large. 
Roughly speaking, in order to bound the dimension, we would like to obtain that for Lebesgue almost every $r\in f(\Delta)$ we have that $f^{-1}(r)$ does not intersect triangles with high conductivity on the $n$th level for large $n$. 
Consequently, by Lemma \ref{approxlevelconductivity} we could deduce that $f^{-1}(r)$ intersects ``many'' triangles, which yields ``high'' Hausdorff dimension due to the Mass Distribution Principle (Theorem \ref{thMDP}). 
In order to formalize this idea, we would like to estimate the number of triangles with high conductivity. 

For any $T\in \tau_{n}^l$ we can consider the chain of triangles $T_1, T_2,\ldots, T_n$ such that $T_i\supseteq T$ and $T_i\in\tau_{i}^l$. 
We bound from above the number of triangles in $\tau_{n}^l$, whose conductivity is at least $2^{-nd_1}$, where $0<d_1\le\frac12$ is chosen to be a small rational number,
 hence $nd_1$ is an integer for infinitely many $n$. From this point on we restrict our arguments to such $n$s,  that is we suppose that $n=n'\mathbf{q}$  for some $n'= 1, 2, ...$ where $\mathbf{q}=\min\{m\in\N : md_1\in\N\}$.
 The conductivity is at least $2^{-nd_1}$ if $T_i$ is an extreme triangle for at least $n- nd_1 $ of the indices $i=1, 2, ..., n$. 

The number of such triangles is estimated from above by
\begin{displaymath}
\binom{n}{n-nd_1}\left(3\left(2^l-1\right)\right)^{nd_1} 2^{n-d_1n}=
\binom{n}{nd_1}\left(3\left(2^l-1\right)\right)^{nd_1} 2^{n-d_1n},
\end{displaymath} 
as we can choose the $n-nd_1$ places where we use one of the two extreme triangles, and in the remaining places we allow the usage of any of the $3\left(2^l-1\right)$ triangles, hence giving an upper bound.
By standard bounds on binomial coefficients, this can be estimated from above by
\begin{equation}
\label{countwellconduct}
\left(\frac{en}{nd_1}\right)^{nd_1}\left(3\left(2^l-1\right)\right)^{nd_1}2^{n-d_1n}.
\end{equation}
The diameter of the triangles in $\tau_{n}^l$ is $2^{-ln}$. 
Consequently, due to $f$ being 1-H\"older-$\alpha$ and by (\ref{countwellconduct}), we know that the $f$-image of the union of the well conducting triangles has 
Lebesgue measure at most
\begin{equation}
\label{imagewellconduct}
\left(\frac{e}{d_1}\right)^{nd_1}\left(3\left(2^l-1\right)\right)^{nd_1}2^{n-d_1n}2^{-ln\alpha} =  \left(\left(\frac{e}{d_1}\right)^{d_1}\left(3\left(2^l-1\right)\right)^{d_1}2^{1-d_1-l\alpha}\right)^n  =: c^n. 
\end{equation}
Assume that  $c<1$. 
Then the corresponding series is convergent, hence we can apply the Borel--Cantelli lemma to deduce that almost every $r\in \conv(f(V(T))$ appears in the image of well conducting triangles only on finitely many levels. 
Consequently, for almost every $r$, if $n$ is large enough, $f^{-1}(r)$ must intersect at least $2^{nd_1}$ triangles of $\tau_n^l$, as the sum of the conductivities of triangles in $T\in\tau_{n}^l$ for which $r\in \conv (f(V(T)))$, is at least 1. We will use this observation paired with the Mass Distribution Principle,
Theorem \ref{thMDP}
to give a lower bound on the dimension of almost every level set, but first, let us consider the question how to choose $l,d_1$ in order to guarantee that $c<1$. Elaborating (\ref{imagewellconduct}), we would like to assure 
\begin{equation}
\label{imagewellconductineq}
\left(\frac{e}{d_1}\right)^{d_1}\left(3\left(2^l-1\right)\right)^{d_1}2^{1-d_1}2^{-l\alpha}<1.
\end{equation}
If this inequality holds for $2^l$ instead of $2^l-1$, that is still fine for our purposes. 
Rewriting our powers in base $e$, it leads to
\begin{displaymath}
\exp\left(d_1 - d_1\log{d_1} + d_1\log{3}+d_1l\log{2}+\log{2}-d_1\log{2}-\alpha l\log{2}\right)<1,
\end{displaymath}
that is after taking logarithm
\begin{displaymath}
d_1(1-\log{d_1} +\log 3 - \log 2) + \log 2 +l  (d_1- \alpha) \log 2  <0.
\end{displaymath}
We clearly need $d_1<\alpha$ to satisfy this inequality, as only the third term can be negative. 
Fixing this assumption, after rearrangement we obtain that it holds if and only if
\begin{equation}
\label{lcondition}
\frac{d_1(1-\log{d_1}+\log{3}-\log{2})+\log 2}{(\alpha-d_1)\log 2} = \frac{d_1(1+\log{\frac{3}{2d_1}})+\log 2}{(\alpha-d_1)\log 2}<l.
\end{equation}
No matter how we fix the rational number $0<d_1<\alpha$, such an $l$ implies $c<1$. 
We notice that $d_1$ can be chosen arbitrarily close to $\frac{\alpha}{2}$, and due to the continuity of the left hand side of  (\ref{lcondition}), if they are sufficiently close to each other, then we can choose $l$ so that
\begin{equation}
\label{lchoice}
\frac{\frac{\alpha}{2}(1+\log{\frac{3}{\alpha}})+\log 2}{\frac{\alpha}{2}\log 2}<l\leq 1+\frac{\frac{\alpha}{2}(1+\log{\frac{3}{\alpha}})+\log 2}{\frac{\alpha}{2}\log 2}.
\end{equation}

We recall that for such $l,d_1$ we have that for almost every $r$, if $n$ is large enough, $f^{-1}(r)$ can only intersect triangles of $\tau_n^l$ with conductivity smaller than $2^{-nd_1}$. 
Fix such an $r$ and consider only such large enough $n$s. 
We define a probability measure $\mu$ on $\Delta^l$. 

Due to Kolmogorov's extension theorem (see for example \cite{Oksendal}, \cite{Taomeas} or \cite{Lamperti}) it suffices to define consistently $\mu(T\cap\Delta^l)$ for any triangle $T$ in $\tau_n^l$. 
First, if $T$ is not an $l,r$-descendant of $\Delta_0$, let $\mu(T\cap\Delta^l)=0$. 
For descendants, we proceed by recursion. 
Notably, if $T$ is an $l,r$-descendant in $\tau_{n}^l$, and $\mu(T\cap\Delta)$ is already defined, then we divide its measure among its $l,r$-descendants in $\tau_{n+1}^l$ proportionally to their conductivity. 
More explicitly, for an $l,r$-descendant $T^*\in\tau_{n+1}^l$ of $T$ we define
\[\mu(T^*\cap\Delta^l)=\mu(T\cap\Delta^l)\frac{\kappa_{n+1}^l(T^*)}{\sum_{T'\in \tau_{n+1}^l,\ T' \text{ is an $l,r$-descendant of } T}\kappa_{n+1}^l(T')}.\]
Then  using Lemma \ref{approxlevelconductivity}  by induction it is clear that 
\[\mu(T\cap\Delta^l)\leq\kappa_{n}^l(T)
\text{  for any $l,r$-descendant $\Delta_0$.}
\]
 
Hence,
\begin{equation}
\label{gentrianglemeasure}
\mu(T\cap\Delta^l)\leq 2^{-nd_1}.
\end{equation} 
Next we want to use the Mass Distribution Principle. Recall that we assumed
that we work with $n$s of the form $n'\mathbf{q}$. 
Now assume that we have a Borel set $U\in\Delta^l$ such that for its diameter we have $2^{-n'\mathbf{q}l}\leq |U| \leq 2^{-(n'-1)\mathbf{q}l}$. 
By a simple geometric argument one can show that $U$ might intersect at most $C$ triangles in $\tau_{n'\mathbf{q}}^l$ for some constant $C$ not depending on $n'$. 
(One can consider the triangular lattice formed by triangles with side length $2^{-ln'\mathbf{q}}$ and it is easy to see that a 
Borel set with diameter $2^{-(n'-1)\mathbf{q}l}$ can intersect only a limited number of the triangles.) Consequently, the number of $l,r$-descendants of $\Delta_0$ in $\tau_{n'\mathbf{q}}^l$  intersected by $U$ is also bounded by $C$. 
For such an $l,r$-descendant $T$ we can apply (\ref{gentrianglemeasure}), hence
\[\mu(U)\leq 2^{-n'\mathbf{q}d_1}C.\]
As $|U|\geq{2^{-n'\mathbf{q}l}}$, the mass distribution principle tells us that if there exists $C',\ s>0$ independent of $n'$ with
\[2^{-n'\mathbf{q}d_1}C\leq \left(2^{-n'\mathbf{q}l}\right)^s C',\]
then $s\leq \dim_H(f^{-1}(r))$. 
Such a $C'$ exists if and only if
\[s\leq \frac{d_1}{l}.\]
Hence the expression on the right hand side of this inequality is a good choice for $s$ in the mass distribution principle, thus it is a lower estimate for $\dim_H(f^{-1}(r))$ for any valid pair $l,d_1$. 
Using (\ref{lchoice}) and the argument leading to it, we can approximate $\frac{\alpha}{2}$ by possible $d_1$s and for sufficiently good approximations we can use
\begin{displaymath}
l\leq 1+\frac{\frac{\alpha}{2}(1+\log{\frac{3}{\alpha}})+\log 2}{\frac{\alpha}{2}\log 2}=1+\frac{1+\log{\frac{3}{\alpha}}}{\log 2}+\frac{2}{\alpha}.
\end{displaymath}
Consequently,
\[\dim_H(f^{-1}(r))\geq \frac{\frac{\alpha}{2}}{1+\frac{1+\log{\frac{3}{\alpha}}}{\log 2}+\frac{2}{\alpha}}>0.\]
\end{proof}

\section{Upper estimate of $D_{*}(\alpha, \Delta)$ for generic functions}\label{*secsierb}

\begin{definition}\label{*pafsier}
We say that $f:\DDD\to \R$ is a {\it piecewise affine function at level $n\in \N$
on the Sierpi\'nski triangle} if it is affine on  any $T\in\tttt_{n}$.

If a piecewise affine function at level $n\in \N$ on the Sierpi\'nski triangle  satisfies the property that for any $T\in\tttt_{n}$ one can always find two vertices of $T$
where $f$ takes the same value, then we say that $f$ is a {\it standard piecewise affine function at level $n\in \N$
on the Sierpi\'nski triangle}.

A function $f:\DDD\to \R$ is a {\it strongly piecewise affine function
on the Sierpi\'nski triangle} if there is an $n\in \N$ such that it is a piecewise affine function at level $n$.
\end{definition}

Here we state a specialized version of Lemma \ref{piecewise_affine_approx}
valid for the Sierpi\'nski triangle.

\begin{lemma} \label{piecewise_affine_approxsier}
Assume that $0<\alpha<1$, and $0<c$ are fixed. 
Then the locally non-constant standard strongly piecewise affine $c^{-}$-H\"older-$\alpha$ functions defined on $\DDD$ form a dense subset of the $c$-H\"older-$\alpha$ functions.
\end{lemma}

Before proving this lemma we need to state and prove another one.

 Recall that  ${\mathbf V}_{n}=\bigcup_{T\in\tau_n} V(T)$. 

\begin{lemma}\label{*lemsiergf}
Suppose, $0<\eps$, $0<\aaa<1$, $0<c$, $f:\DDD\to\R$ is Lipschitz-$M$ and
$c^{-}$-H\"older-$\alpha$ on $\DDD$. Then there exists $N\in\N$
 such that for any fixed $n\geq N$
 if for any $T\in\tau_n$
 $g$ is $c/8$-H\"older-$\aaa$ on $T\cap\DDD$ and 
$g(x)=f(x)$ for all $x\in {\mathbf V}_{n}$
then
\begin{equation}\label{*eqsiergf}
||g-f||_{\oo}<\eee\text{ and $g$ is $c^{-}$-H\"older-$\alpha$ in $\DDD$}.
\end{equation}
\end{lemma}

\begin{proof}
Since $f$ is $c^{-}$-H\"older-$\alpha$ on $\DDD$ we can choose $0<c'<c$  such that 
$f$ is $c'$-H\"older-$\alpha$ on $\DDD$.

If $$|x-y|<\Big (\frac{c}{16 M} \Big )^{\frac{1}{1-\aaa}}=:\overline{M}$$
then
$$|f(x)-f(y)|\leq M|x-y|^{1-\aaa}|x-y|^{\aaa}<\frac{c}{16}|x-y|^{\aaa}.$$
Choose $c''\in(c',c)$.

First we prove that $g$ is $c^-$-H\"older-$\alpha$. 

Suppose that  $|x-y|\geq \overline{M} /4$, 
$x\in T_{x}\in \tttt_{n}$,  and $y\in T_{y}\in \tttt_{n}$ and select vertices
\begin{equation} \label{*eqvert}
\text{ $ \nu_x \in V(T_{x})$ and $ \nu_y \in V(T_{y})$.}
\end{equation}
 Then by our assumption $f( \nu_x )=g( \nu_x )$
and $f( \nu_y )=g( \nu_y )$.
Since the diameter of $T_{x}$ and $T_{y}$ equals $2^{-n}$
we obtain 
\begin{equation*}
\begin{split}
|g(x)-g(y)|&\leq |g(x)-g( \nu_x )|+|g( \nu_x )-g( \nu_y )|+|g( \nu_y )-g(y)| \\
&\leq 2c 2^{-n\aaa}+|f( \nu_x )-f( \nu_y )| 
\le 2c 2^{-n\aaa}+c'| \nu_x - \nu_y |^{\aaa} \\
&\le 2c 2^{-n\aaa}+c'\big(|x-y| +2\cdot2^{-n}\big)^{\aaa}  \xrightarrow[n\to\infty]{} c' |x-y|^{\aaa},
\end{split}
\end{equation*}
 where the convergence is uniform due to $|x-y|$ being separated from zero.  Thus, we can choose $N$ large enough (independently of $x$ and $y$) such that
\begin{equation}\label{*gestb}
|g(x)-g(y)| \le c'' |x-y|^{\aaa}.
\end{equation}

Suppose $|x-y|< \overline{M}/4$.
If $T_{x}=T_{y}$ then by our assumption
\begin{equation}\label{*gtxty}
|g(x)-g(y)|< \frac{c}{8}|x-y|^{\aaa}.
\end{equation}

If $T_{x}\not=T_{y}$, but $T_{x}$ and $T_{y}$ has a common vertex $v$
then by geometric properties of the Sierpi\'nski triangle $xvy\sphericalangle\ge \pi/6$, hence by the Law of sines
$$|x-v|=\frac{|x-y |\sin(vyx \sphericalangle) }{\sin (xvy \sphericalangle)}\leq 
|x-y|\frac{2}{\sqrt 3}$$
and similarly
$|y-v|\leq |x-y|\frac{2}{\sqrt 3}$.
Hence,
\begin{equation}\label{*gxyv}
\begin{split}
|g(x)-g(y)|&\leq |g(x)-g(v)|+|g(v)-g(y)|\leq \frac{c}{8}|x-v|^{\aaa}+\frac{c}{8} |v-y|^{\aaa}\\
 &\le 2\frac{c}{8}\Big (\frac{2}{\sqrt 3} \Big )^{\aaa}|x-y|^{\aaa}< \frac{c}{2}|x-y|^{\aaa}.
\end{split}
\end{equation}

If $T_{x}$ and $T_{y}$ does not have a common vertex then
$|x-y|\geq 2^{-n} \frac{\sqrt 3}{2} $.
Choose $v_x\in V(T_{x})$ and $v_y\in V(T_{y})$. Then 
$$
|v_x-v_y |\leq |x-y|+2\cdot 2^{-n}\leq |x-y|\Big(1+2\cdot \frac{2}{\sqrt 3}\Big)< 4|x-y| .
$$ 
Thus
\begin{equation}\label{*gtxyb}
\begin{gathered}
|g(x)-g(y)|\leq |g(x)-g(v_x)|+|g(v_x)-g(v_y)|+|g(v_y)-g(y)|\\
\leq 2\frac{c}{8}
2^{-n\aaa}+|f(v_x)-f(v_y)|
<2\frac{c}{8}  \Big ( \frac{2}{\sqrt 3} \Big )^{\aaa} |x-y|^{\aaa}+ \frac{c}{16} |v_x-v_y|^{\aaa} \\
< 2\frac{c}{8} \Big ( \frac{2}{\sqrt 3} \Big )^{\aaa}|x-y|^{\aaa}+ \frac{c}{16} 4^{\aaa}|x-y|^{\aaa}  
\leq c \Big ( \frac{1}{2\sqrt 3}+\frac{1}{4} \Big )|x-y|^{\aaa}<\frac{c}{\sqrt 3}|x-y|^{\aaa}.
\end{gathered}
\end{equation}

From \eqref{*gestb}, \eqref{*gtxty}, \eqref{*gxyv} and \eqref{*gtxyb} it follows that
$g$ is $c^{-}$-H\"older-$\alpha$ on $\DDD$. 

To see  the inequality  in \eqref{*eqsiergf} select $v_x \in V(T_{x})$.
Then
$$|f(x)-g(x)|\leq |f(x)-f(v_x)|+|f(v_x)-g(v_x)|+|g(x)-g(v_x)|< M|x-v_x|+ 0 $$
$$ + \frac{c}{8}|x-v_{x}|^{\aaa}
 \leq M\cdot 2^{-n}+\frac{c}{8}2^{-n\aaa}
\leq M\cdot 2^{-N}+\frac{c}{8}2^{-N\aaa}<\eee,$$
if $N$ is chosen sufficiently large.
\end{proof}

\begin{proof}[Proof of Lemma \ref{piecewise_affine_approxsier} .]
With a  rather straightforward  modification of the proof of Lemma \ref{piecewise_affine_approx}
one  can verify  
the following weaker form of Lemma \ref{piecewise_affine_approxsier}: locally non-constant strongly piecewise affine $c^{-}$-H\"older-$\alpha$ functions defined on $\DDD$ form a dense subset of the $c$-H\"older-$\alpha$ functions.

Hence suppose that $f\in C_{c}^{\aaa}(\DDD)$ and $\eee>0$ are given.
By the previous remark choose a locally non-constant $f_{1}\in C_{c^{-}}^{\aaa}(\DDD)$
and $n\in \N$ such that 
\begin{equation}\label{*foapp}
||f-f_{1}||_{\oo}<\eee/2 \text{ and $f_{1}$ is piecewise affine  on  any }T\in\tttt_{n}.
\end{equation}
By our assumption about the diameter of $\DDD$, the triangles in $\tttt_{n}$
are of side length $2^{-n}$.

Since $f_{1}$ is piecewise affine  on any $T\in \tau_n$ it  is Lipschitz-$M$ for a suitable $M$ and $c^{-}$-H\"older-$\alpha$
on $\DDD.$
By Lemma \ref{*lemsiergf} used with $f_{1}$ and $\eee/2$ instead of $f$ and $\eee$ choose $N$.

We want  to  obtain a locally non-constant {\it standard} strongly piecewise affine $c^{-}$-H\"older-$\alpha$ function $f_{2}$ which is $\eee$-close to $f$.
Since $f_{1}$ is a locally non-constant strongly piecewise affine $c^{-}$-H\"older-$\alpha$ function which is $\eee/2$-close to $f$
we will modify $f_{1}$ to obtain $f_{2}$, $\eee$-close to $f$.

Select a sufficiently large $n'\geq \max\{ N,n \}$
which satisfies
\begin{equation}\label{*npas}
\frac{4}{\sqrt 3}  M  (2^{-n'})^{1-\aaa}<\frac{c}{8}.
\end{equation}
To obtain $f_{2}$ we will modify $f_{1}$ on the triangles $T\in \tttt_{n'}$.
On ${\mathbf V}_{n'}$ the functions  $f_{2}$ and $f_{1}$ will coincide.

Suppose that $T\in \tttt_{n'}$ is arbitrary. Denote its vertices by $v_{1},\ v_{2}$
and $v_{3}$. Suppose that $v_{4},\ v_{5}$ and $v_{6}$ are the midpoints of the 
segments $v_{1}v_{2}$, $v_{2}v_{3}$ and $v_{1}v_{3}$, respectively.
We denote by $T_1$, $T_{2}$ and $T_{3}$ the triangles  $v_{1}v_{4}v_{6}$,
$v_{4}v_{2}v_{5}$  and $v_{5}v_{3}v_{6}$, respectively.
The triangles $T_{j}$, $j=1,2,3$ belong to $\tttt_{n'+1}$.

We define $f_{2}(v_{1})=f_{2}(v_{4})=f_{1}(v_{1})$,
$f_{2}(v_{2})=f_{2}(v_{5})=f_{1}(v_{2})$
and $f_{2}(v_{3})=f_{2}(v_{6})=f_{1}(v_{3})$.
We also assume that $f_{2}$ is affine on any triangle $T'\in \tttt_{n'+1}$.

By our choice of $M$ we have
$$|f_{1}(v_{i})-f_{1}(v_{j})|\leq  M \cdot  2^{-n'} \text{ for any }i,j\in \{ 1,2,3 \}.$$  
Suppose that $x,y\in T_{1}\cap \DDD$ (a similar argument works for the triangles $T_{2}$ and $T_{3}$).
Denote by $\pi$ the orthogonal projection onto the  second coordinate ``$\mathbf y$''-axis  
then
$$|f_{2}(x)-f_{2}(y)|\leq \frac{ \left|  f_{1}(v_{3})-f_{1}(v_{1})\right|}{\frac{\sqrt 3}{2} \cdot 2^{-n'-1}}|\pi(x)-\pi(y)|  \leq  \frac{4}{\sqrt 3} M |x-y|$$ 
$$= \Big ( \frac{4}{\sqrt 3} M |x-y|^{1-\aaa} \Big )|x-y|^{\aaa}\leq
\frac{4}{\sqrt 3}  M (2^{-n'})^{1-\aaa} |x-y|^{\aaa}< \frac{c}{8}|x-y|^{\aaa}, $$
where at the last step we used \eqref{*npas}.
Hence if Lemma \ref{*lemsiergf} is applied with the constants fixed above to the function $f_{2}$ as $g$, we obtain a standard strongly piecewise affine $c^{-}$-H\"older-$\alpha$ function
which is $\eee$-close to $f$. 
 \end{proof}

We denote by $\DDD^{*}$ the rescaled and translated copy of $\DDD$
in a way that the vertices of $\DDD^{*}$ are $v^{*}_{1}=(0,0)$, 
$v^{*}_{2}=(2/\sqrt 3,0)$ and $v^{*}_{3}=(1/\sqrt 3 ,1)$.

It is clear that $\R^{2}$ can be tiled by translated copies of the triangle  $v^{*}_{1}v^{*}_{2}v^{*}_{3}$ and its mirror image about the $x$-axis. We denote the system of these triangles by ${\mathbf T}^{*}_{0}$.
For $n\in \N$  we also  use the scaled copies of this triangular tiling consisting of triangles
of the form $2^{-n}T$, $T\in {\mathbf T}^{*}_{0}$. The system of triangles belonging to this tiling is denoted by ${\mathbf T}^{*}_{n}$.
During the definition of box dimension many different concepts can be used,
see for example \cite{[Fa1]}.
Given a set $F\sse \R^{2}$ we denote by 
${\mathbf N}^{*} _{n}(F)$ the number of those triangles $T\in {\mathbf T}^{*}_{n}$ which intersect $F$.
It is an easy exercise to see that
\begin{equation}\label{*bttcsbdim}
\udimb F=\limsup_{n\to\oo} \frac{\log {\mathbf N}^{*}_{n}(F)}{n\log 2}\text{ and }
\ldimb F=\liminf_{n\to\oo} \frac{\log {\mathbf N}^{*}_{n}(F)}{n\log 2}.
\end{equation}

\begin{lemma}\label{*lemfab}
Suppose $0<\aaa<1$. There exists $\fff:\DDD^{*}\to [0,1]$,
$\fff\in C_{3}^{\aaa}(\DDD^{*})$  such that 
$\fff(v^{*}_{1})=\fff(v^{*}_{2})=0$, $\fff(v^{*}_{3})=1$
and  there exists an exceptional set ${\mathbf E}^{*}$ such that $\lll({\mathbf E}^{*})=0$
and for any  $y\in \R\sm {\mathbf E}^{*}$
\begin{equation}\label{*lemfabconc}
\udimb \fff^{-1}(y)\leq 1-2^{-\aaa},\text{ that is }\limsup_{n\to\oo} \frac{\log {\mathbf N}^{*}_{n}(\fff^{-1}(y))}{n\log 2} \leq 1-2^{-\aaa}.
\end{equation}
Since the lower, and hence the upper box dimension is never less than the Hausdorff dimension we also have
\begin{equation}\label{*lemfabconcb}
\dimh \fff^{-1}(y)\leq 1-2^{-\aaa}.
\end{equation}
\end{lemma}

\begin{figure}[ht]
\includegraphics[width=1.0\textwidth]{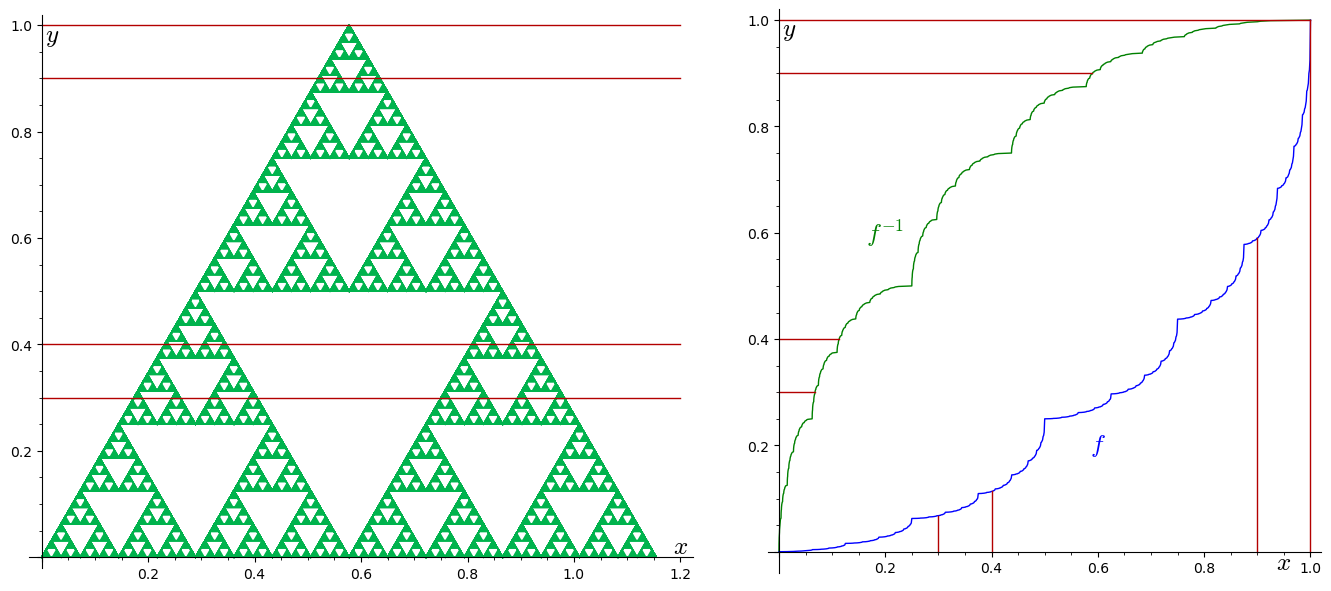}
\caption{the rescaled and translated Sierpi\'nski triangle $\DDD^{*}$, the function $f(x)=\lambda_{p}([0,x))$ and its inverse $f^{-1}$}
\end{figure}

\begin{proof} The basic concepts and results of ergodic theory we use in the sequel can be found for example in \cite{KaHa}.
Suppose that $1/2<p<1$. Denote by $\sss$ the doubling map on $[0,1)$,
that is $\sss(x)=\{ 2x \}$, where $\{ . \}$, denotes the fractional part.
Denote by $\lll_{p}$, the   $\sigma$  invariant ergodic measure for which
\begin{equation}\label{*lpdef}
\lambda_{p} \Big ( \sum_{k=1}^{n} {\mathbf e}_{k}2^{-k}+2^{-n}[0,1) \Big )=p^{\sum_{k=1}^{n}{\mathbf e}_{k}}(1-p)^{\sum_{k=1}^{n}(1-{\mathbf e}_{k})}\text{, where } {\mathbf e}_{k}\in \{ 0,1 \}, \ k=1,2,... .
\end{equation}

Set $f(x)=\lambda_{p}([0,x))$. 
Suppose that $x,y\in [0,1]$, $x< y$ and $2^{-n}\leq y-x<2^{-n+1}$, $k\in\N$.
By \eqref{*lpdef} any interval of the form
$\sum_{k=1}^{n} {\mathbf e}_{k}2^{-k}+2^{-n}[0,1)$ is of $\lambda_{p}$ measure at most $p^n$.
Since $[x,y]$ can be covered by no more than three such intervals we have
\begin{equation}\label{*fholderexp}
|f(x)-f(y)|\leq 3 \cdot p^{n}= 3\cdot 2^{n\log_{2} p}\leq  3 |x-y|^{-\log_{2} p}. 
\end{equation}

Since $\sss$ is ergodic with respect to $\lll_{p}$ by  the Birkhoff Ergodic Theorem (see for example Chapter 4 of \cite{KaHa})  we have for $\lll_{p}$
almost every $x$
\begin{equation}\label{*ergthl}
\frac{\sum_{k=1}^{n}{\mathbf e}_{k}(x)}{n}=\frac{\sum_{k=1}^{n}\chi_{[1/2,1)}(\sss^{k}x)}{n}\to \lambda_{p}([1/2,1))=p
\end{equation}
(where $\mathbf{e}_k(x)$ denotes the $k$th digit after the binary point in the binary representation of $x$).
We denote by $X_p$ the set of  $x$s  satisfying \eqref{*ergthl}.
Since for any $x,y\in [0,1]$, $x< y$ we have $\lll(f([x,y)))=\lambda_{p}([0,y))-\lambda_{p}([0,x))=\lambda_{p}([x,y))$ and the intervals $[x,y)$ generate the Borel sigma algebra we have
$\lambda_{p}(A)=\lll(f(A))$ for any Borel set $A\sse [0,1)$.
Hence $\lll(f(X_{p}))=1$ and for $\lll$ almost every $y\in [0,1]$ we have $f^{-1}(y)\in X_{p}.$

For $(x,y)\in \DDD^{*}$ set $\fff(x,y)=f(y)$ and select $p$ such that $-\log_{2} p=\aaa$.  From \eqref{*fholderexp}
it follows that $\fff$ is a $3$-H\"older-$ \aaa$ function. 

The definition of the Sierpi\'nski  triangle implies that if $y=\sum_{k=1}^{\oo}{\mathbf e}_{k} 2^{-k}$ 
is not a dyadic rational then the horizontal line 
\begin{equation}\label{*hint}
\text{$\ds \{ (x,y):x\in \R \}$ intersects $ 2^{\sum_{k=1}^{n}(1-{\mathbf e}_{k})}$ many triangles  of ${\mathbf T}^{*}_{n}$.
}
\end{equation}
If needed, by removing a countable set we can assume that $f(X_{p})$ does not contain dyadic rational numbers.
Set ${\mathbf E}^{*}= [0,1]\sm f(X_p)$. If $y\not \in [0,1]$ then $\fff^{-1}(y)=\ess$
and \eqref{*lemfabconc} is obvious.
If $y\in [0,1]\sm {\mathbf E}^{*}$ then by \eqref{*ergthl} and \eqref{*hint} we have
\begin{equation}\label{*Nnas}
\limsup_{n\to\oo} \frac{\log {\mathbf N}^{*}_{n}(\fff^{-1}(y))}{n\log 2}= 
\limsup_{n\to\oo} \frac{\log 2^{\sum_{k=1}^{n}(1-{\mathbf e}_{k})}}{-\log 2^{-n}}
\end{equation}
$$=
\lim_{n\to \oo} 1-\frac{\sum_{k=1}^{n}{\mathbf e}_{k}(x)}{n}=1-p=
1-2^{-\aaa}.
$$
\end{proof}

\begin{theorem}\label{*thsierfb}
For any $0<\alpha< 1$, we have $D_{*}(\alpha, \Delta)\leq 1-2^{-\aaa}$.
\end{theorem}

In the proof we will use the following Corollary of Lemma 5.1 from \cite{sier}:

\begin{lemma}\label{*lemdfb}
Suppose that   $0< \aaa\leq 1$,   $F\subset\R^p$ is compact.
If $\{f_1,f_2,\ldots\}$ is a countable dense subset of $C_1^\alpha(F)$, then there is a dense $G_\delta$ subset $\cag$ of $C_1^\alpha(F)$ such that 
\begin{equation}\label{eqdfb}
\sup_{f\in\cag} D_*^f(F) \le \sup_{k\in\N} D_*^{f_k}(F).
\end{equation}
\end{lemma}



\begin{proof}[Proof of Theorem \ref{*thsierfb}.]
Suppose that $\{g_{k}:k\in\N  \}$ consists  of  locally non-constant standard strongly piecewise affine $1^{-}$-H\"older-$\alpha$ functions defined on $\DDD$,
and this set is dense in the space of $1$-H\"older-$\alpha$ functions defined on $\DDD$. It is also clear that each $g_{k}$ is Lipschitz with a constant which we denote by $M_{k}$.

We suppose that $n_{k}$ is selected in a way that $g_{k}$ is piecewise affine on each $T\in \tau_{n_k}$  and there exist $v_{1}(T),\ v_{2}(T)\in V(T)$ such that
\begin{equation}\label{*vertice}
\text{$g_{k}(v_{1}(T))=g_{k}(v_{2}(T))$ and if $v_{3}(T)$ denotes the third vertex of $T$
}
\end{equation}
$$\text{then $g_{k}(v_{3}(T))\not=g_{k}(v_{1}(T))$.}$$
Observe that if we take subdivisions, that is we take an $n_{k}'\geq n_{k}$
then \eqref{*vertice} holds for suitably chosen 
vertices of triangles $T\in \tau_{n_{k}'} $.  Later in the proof we will select a 
sufficiently large $n_{k}'$. 

Next we define a function $f_{k}$ satisfying
\begin{equation}\label{*gkmfk}
||f_{k}-g_{k}||_{\oo}< 2^{-k}.
\end{equation}

First using Lemma \ref{*lemsiergf}  with $\aaa=\aaa$, $c=1$, $f=g_{k}$, $M=M_{k}$ and $\eee=2^{-k}$, select $N_{k}\geq n_{k}$.

We define $f_{k} $  such that with $n_{k}'> N_{k}$ 
\begin{equation}\label{*gkmdefa}
f_{k}(x)=g_{k}(x) \text{ for any  }x\in \bigcup_{T\in\tau_{n_{k}'} } V(T).
\end{equation}

Since $g_{k}$ is $M_{k}$-Lipschitz if $T\in \tau_{n_{k}'} $, $x,y\in V(T)$ are different
then $|x-y|=2^{-n_{k}'}$ and
\begin{equation}\label{*fkxymk}
|g_{k}(x)-g_{k}(y)|\leq M_{k}|x-y|=M_{k}\cdot 2^{-n_{k}'(1-\aaa)}|x-y|^{\aaa}<
\frac{1}{100}|x-y|^{\aaa},
\end{equation}
if we suppose that $n_{k}'$ is chosen large enough to satisfy
\begin{equation}\label{*nkmch}
M_{k}\cdot 2^{-n_{k}'(1-\aaa)}< \frac{1}{100}.
\end{equation}

Suppose that $T\in \tau_{n_{k}'} $.
For ease of notation we will write
$v_{i}$ instead of $v_{i}(T)$ for $i=1,2,3.$ 
Using notation from the second paragraph before Lemma \ref{*lemfab}, denote by $\Psi_{T}$  the similarity for which $\Psi_{T}(T\cap \DDD)=\DDD^{*}$ and
the vertices of $T$ for which we have \eqref{*vertice} satisfied are mapped in a way that 
\begin{equation}\label{*vi}
\text{$\Psi_{T}(v_{i})=v_{i}^{*}$ for $i=1,2,3.$}
\end{equation}
Then for every $x,y\in\Delta^*$
\begin{equation}\label{*PPPT}
|\Psi_{T}(x)-\Psi_{T}(y)|=2^{n_{k}'}\frac{2}{\sqrt 3}|x-y|.
\end{equation}

Let $\fff\in C_{3}^{\aaa}(\DDD^{*})$ be given by Lemma \ref{*lemfab}.
For $x\in T\cap \DDD$ we put 
\begin{equation}\label{*gkmx}
f_{k}(x)=\fff(\Psi_{T}(x))(g_{k}(v_{3})-g_{k}(v_{1}))+g_{k}(v_{1}).
\end{equation}
Suppose $x,y\in T\cap \DDD$ then  
\begin{equation}\label{*gkmh}
|f_{k}(x)-f_{k}(y)|\leq |g_{k}(v_{3})-g_{k}(v_{1})|\cdot 
3|\Psi_{T}(x)-\PPP_T(y)|^{\aaa}
\end{equation}
$$\leq M_{k}2^{-n_{k}'}\cdot 3\cdot 2^{n_{k}'\aaa}
\cdot \Big ( \frac{2}{\sqrt 3} \Big )^{\aaa}|x-y|^{\aaa}=M_{k}2^{-n_{k}'(1-\aaa)}\cdot 3
\cdot \Big ( \frac{2}{\sqrt 3} \Big )^{\aaa}|x-y|^{\aaa}<\frac{1}{8}|x-y|^{\aaa},$$
where at the last step we used \eqref{*nkmch}.

Now by \eqref{*gkmdefa}, \eqref{*gkmh} and $n_{k}'\geq N_{k}$ we can apply Lemma \ref{*lemsiergf}. Thus $f_{k}$ is a $1^-$-H\"older-$\aaa$ function satisfying \eqref{*gkmfk}.

Since $\tau_{n_{k}'}$ consists of finitely many triangles $T$, finite union of exceptional sets of measure zero is still of measure zero, and affine transformations are not changing the Hausdorff dimension
we obtain from \eqref{*lemfabconcb} and \eqref{*gkmx} that 
$\dimh f_{k}^{-1}(y)\leq 1-2^{-\aaa}$
for almost every $y\in \R$ and for every $k$.
Therefore $D_*^{f_k}(\DDD )\leq 1-2^{-\aaa}$ for every $k$ and by
the density of the functions $g_{k}$ and \eqref{*gkmfk} the functions 
$f_{k}$ are also dense in $C_1^\alpha(\DDD)$.
Hence we can apply Lemma \ref{*lemdfb} with the compact set $\DDD$ and the dense set of functions
$f_{k}$ to obtain a dense $G_{\ddd}$ set $\cag_{1}$ such that 
$D_*^{f}(\DDD )\leq 1-2^{-\aaa}$
for any $f\in \cag_{1}$. 
Since in \eqref{*defDcsaF} there is also a supremum a little extra care is needed.
Using Theorem \ref{*thmgenex} select and denote by $\cag_{2}$  a dense $G_{\ddd}$
subset
of $C_1^\alpha(\DDD )$ such that for every $f\in\cag_{2}$ we have $D_*^f(\DDD ) = D_*(\alpha,\DDD )$.
Since $\cag_{1}\cap\cag_{2}$ is non-empty we can select a function
$f$ from it. For this function $D_*^{f}(\DDD )= D_*(\alpha,\DDD )\leq 1-2^{-\aaa}.$
This completes the proof of the theorem.
\end{proof}

 For $\aaa<1$  from Theorem \ref{thm:trivial_upper_bound}
one can obtain that $D_{*}(\aaa, \Delta)\leq \frac{\log 3}{\log 2}-1\approx 0.584962500721.$
 Since $\lim_{\aaa\to 1-0}1-2^{\aaa}=1/2$ this upper estimate is worse.


\section{Strongly separated fractals}\label{*secssc}

In this section, our goal is to prove that $D_{*}(\alpha, F)$ vanishes for small $\alpha$ in the case when $F$ admits a strongly separated structure in the following sense:

\begin{definition}\label{*defssepf}
For some $0<\nu, \rho<1$, a nonempty set $F \subseteq \R^p$ admits a $(\nu, \rho)$ separated structure, if there exists  $K>0$, and a sequence of finite families $\mathcal{S}_k$ such that 
\begin{itemize}
\item $ F\subset\bigcup \mathcal{S}_k$ for each $k$,
\item for any $k$ and $ F'\in \mathcal{S}_k$ we have $| F'|<K\nu^k$,
\item for any $k$ and  distinct $F_{i}, F_j \in \mathcal{S}_k$  we have $\frac{1}{K}\rho^k< d(F_i, F_j)=\inf\{ |x-y|: x\in F_{i},\ y\in F_{j} \}$.
\end{itemize}
\end{definition}

As the next lemma shows such fractals are quite common.
For the well-known definitions of iterated function systems and the strong separation condition we refer to 
standard textbooks on fractal geometry like \cite{[Fa1]}.

\begin{lemma} \label{lemma:biLipIFS}
Assume that $f_1, ..., f_m$ is an iterated function system satisfying the strong separation condition. Moreover, assume that each $f_i$ is bi-Lipschitz, that is for $1\leq i \leq m$ there exists $0<\rho_i, \nu_i<1$ with
$$\nu_i |x-y| \leq |f_i(x)-f_i(y)| \leq \rho_i |x-y|.$$
Then the attractor $F$ of the system admits a $(\nu,\rho)$ separated structure for some $\nu,\rho>0$. 

More specifically, if each $f_i$ is a similarity, that is $F$ is a self-similar set, then $F$ admits a $(\nu,\nu)$ separated structure for some $\nu>0$.
\end{lemma}

\begin{proof}
For any $j_1, ..., j_k \in \{1, 2, ..., m\}$ we say that $f_{j_k}(...(f_{j_1}(F))...)=F_{j_1 j_2 ... j_k}$ is a $k$th level cylinder of $F$. 

First we show that $\nu = \min_{1\leq i \leq m} \nu_i$ is a valid choice. To establish this, we will define the required families  for any $k$ 
by considering smartly chosen cylinder sets.
Notably, $\mathcal{S}_k$ will consist of cylinders $C_1, ..., C_l$ such that for the diameter $|C_j|$ of any of them,
\begin{equation} \label{diamcondition}
\nu^{k+1}|F| \leq |C_j| \leq \nu^k |F|.
\end{equation} 
This condition is clearly satisfiable by iteratively splitting the cylinders we consider. In particular, start this procedure with the $0$-level cylinder $F$, split it into $m$ 
many first level cylinders. Later on, in each step split precisely those cylinders which have diameter larger than $\nu^k|F|$, and leave the others unchanged. Due to the bi-Lipschitz property of each $f_i$, this algorithm produces a finite system of cylinders in finitely many steps, such that each cylinder satisfies \eqref{diamcondition}. It yields that the above choice of $\nu$ is valid indeed for large enough $K$.

Assume now that the minimal distance between any two of the sets $F_1, F_2, ..., F_m$ is $r$, and consider arbitrary cylinders $C_j, C_l \in \mathcal{S}_k$. Now let $C$ be the smallest cylinder set containing both $C_j$ and $C_l$. In this case, $C=g(F)$, where $g$ is the composition of a finite sequence of functions $f_{i_1}, ..., f_{i_L}$ for some $1\leq i_1, ..., i_L \leq m$. Consequently, $C_j \subseteq g(F_{j'})$ and $C_l \subseteq g(F_{l'})$ for some $1\leq j', l' \leq m$. It yields that the distance between $C_{j}$ and $C_{l}$ is at least as large as the distance between $g(F_{j'})$ and $g(F_{l'})$. 
Moreover, $C$ has diameter at least $\nu^k |F|$: otherwise it would not have been splitted during the procedure creating $\mathcal{S}_k$. 
Hence if $\rho_{*}=\max_{1\leq i \leq m} \rho_i$, we can deduce that for the number $L$ of functions determining $g$ we have 
\begin{align*}
\rho_*^L|F| \ge |g(F)| &\ge \nu^k|F| \\
L\log\rho_* &\ge k\log\nu \\
L &\le \frac{k\log\nu}{\log\rho_*}.
\end{align*}
That is $L\le kL_*$ for
$$L_* = \frac{\log \nu}{\log \rho_{*}}.$$
Altogether it yields that as the distance between $F_{j'}$ and $F_{l'}$ is at least $r$, the distance between $g(F_{j'})$ and $g(F_{l'})$ is at least
$$r\nu^{kL_*} = r \left(\nu^{L_*}\right)^k,$$
which implies that $\nu^{L_*}$ is a valid choice for $\rho$ with a large enough $K$, concluding the proof of the first part.

Concerning the statement for self-similar sets, capitalizing on the fact that $g$ is a similarity, we are able to take a more comfortable route to conclude the proof from the observation that $C$ has diameter at least $\nu^k |F|$. Notably, this implies that the similarity ratio of $g$ is at least $\nu^k$, and consequently, the distance between $C_j$ and $C_l$ cannot be smaller than $r\nu^k$. It verifies that in this case $\rho=\nu$ can be chosen.
\end{proof}

The essence of this section is the following lemma:

\begin{lemma} \label{piecewiseconstapprox}
Assume that $F$ admits a $(\nu,\rho)$ separated structure, and $0<\alpha<\frac{\log \nu}{\log \rho}$.
Then piecewise constant functions with finitely many pieces form a dense subset of the 1-H\"older-$\alpha$ functions.
\end{lemma}

\begin{proof}
Taking union over $0<c<1$, $c$-H\"older-$\alpha$ functions clearly form a dense subset of 1-H\"older-$\alpha$ functions. 
Consequently, it is sufficient to prove that for any $c$-H\"older-$\alpha$ function $f$ we can find a piecewise constant 1-H\"older-$\alpha$ function $\tilde{f}$ in the $\varepsilon$ neighborhood of $f$ in the supremum norm for fixed $\varepsilon>0$. 
To this end, choose $f_r$ according to Lemma \ref{lipschitzapprox} such that it is in the $\frac{\varepsilon}{2}$ neighborhood of $f$,  $M$-Lipschitz and $c$-H\"older-$\alpha$. 
Our aim is to introduce some further perturbation to obtain the 1-H\"older-$\alpha$ function $\tilde{f}$, which is piecewise constant on $F$. We will achieve this goal by using the covers granted by the separated structure of $F$. Notably, we will consider the covering $\mathcal{S}_k=\{F_1, ..., F_l\}$ 
guaranteed by Definition \ref{*defssepf}
  for large enough $k$, and define $\tilde{f}$ separately on $F_1, ..., F_l$ by $\left.\tilde{f}\right|_{F_i} =f_r(x_i)$, using some reference points $x_i\in F_i$. 
Now we would like to prove that the function $\tilde{f}$ is 1-H\"older-$\alpha$ for large enough $k$. 
Choose points $y,y'$ from distinct elements of the covering $\mathcal{S}_k$, where the reference points are $x,x'$. (If $y,y'$ are in the same element of covering, we have nothing to prove.) We have
\begin{displaymath}
|\tilde{f}(y)-\tilde{f}(y')| = |f_r(x)-f_r(x')|.
\end{displaymath}
Then by the triangle inequality, and the H\"older and Lipschitz properties of $f$
\begin{displaymath}
|\tilde{f}(y)-\tilde{f}(y')|\leq |f_r(y)-f_r(y')|+|f_r(y)-f_r(x)|+|f_r(y')-f_r(x')| 
\end{displaymath}
$$\leq c|y-y'|^\alpha +M|y-x|+M|y'-x'|.$$
Hence it is sufficient to prove
\begin{displaymath}
c|y-y'|^\alpha +M|y-x|+M|y'-x'| \leq |y-y'|^\alpha,
\end{displaymath}
that is
\begin{displaymath}
M|y-x|+M|y'-x'| \leq (1-c)|y-y'|^\alpha.
\end{displaymath}
Now on the right hand side $|y-y'|\geq \frac{1}{K}\rho^k$, while on the left hand side both distances are at most $K\nu^k$, where $K$ comes from
Definition \ref{*defssepf}. 
Thus it suffices to prove that for large enough $k$ we have
\begin{displaymath}
2K\cdot M \nu^k \leq \frac{1-c}{K^\alpha}(\rho^k)^\alpha.
\end{displaymath}
However, it immediately follows from the choice of $\alpha$, as that guarantees $\rho^\alpha>\nu$. That is, $\tilde{f}$ is 1-H\"older-$\alpha$  if $k$ is chosen sufficiently large. 
Moreover, by increasing $k$, $f_r$ and $\tilde{f}$ can be arbitrarily close to each other. 
Consequently, $\tilde{f}$ can be in the $\varepsilon$ neighborhood of $f$, which yields the statement of the lemma.
\end{proof}

The following theorem is a straightforward consequence of Lemma \ref{piecewiseconstapprox}:

\begin{theorem} \label{sscdegenerate}
Assume that $F$ admits a $(\nu,\rho)$ separated structure, and $0<\alpha<\frac{\log \nu}{\log \rho}$. Then for the generic 1-H\"older-$\alpha$ function we have that $\lambda(f(F))=0$, and consequently, $\underline{D}_{*}(\alpha,F)=D_{*}(\alpha, F)=0$.
\end{theorem}

\begin{proof}
 Due  to Lemma \ref{piecewiseconstapprox}, the piecewise constant 1-H\"older-$\alpha$ functions form a dense subset of the 1-H\"older-$\alpha$ functions. 
Such a function $f_0$ has a finite range, hence for every $l\in\N$, in a small enough neighborhood of it, for any function $f$ we have $\lambda(f(F))<\frac{1}{l}$. 
By taking the union of all such neighborhoods we find an open, dense subset of the 1-H\"older-$\alpha$ functions, in which $\lambda(f(F))<\frac{1}{l}$. 
Taking intersection 
of these open sets
for $l=1, 2, ...$ we obtain that generically, $\lambda(f(F))=0$.
\end{proof}

Coupling this result with Lemma \ref{lemma:biLipIFS} yields the following corollary:

\begin{corollary} \label{cor:biLipIFS}
If $F$ is the attractor of a bi-Lipschitz iterated function system satisfying the strong separation condition, then for small enough $\alpha>0$ we have $\underline{D}_{*}(\alpha,F)=D_{*}(\alpha, F)=0$.

More specifically, if $F$ is a self-similar set satisfying the strong separation condition, then for $0<\alpha<1$ we have $\underline{D}_{*}(\alpha,F)=D_{*}(\alpha, F)=0$.
\end{corollary}

Taking into consideration Theorem \ref{allexpthm}, we can see that in contrast with certain results of fractal geometry, this corollary does not extend to self-similar sets satisfying the open set condition instead of the strong separation condition.

\section{Phase transition}\label{*secphase}

Looking at the example with the Sierpi\'nski triangle our lower estimate
for $D_{*}(\alpha, \Delta)$ was positive for positive $\aaa$s, and hence 
$D_{*}(\alpha, \Delta)>D_{*}(0, \Delta)=0$. On the other hand, according to Corollary \ref{cor:biLipIFS}, if $F$ is a self-similar set satisfying the strong separation condition, then $D_{*}(\alpha, F)=0$ for $0<\alpha<1$. This phenomenon reflects the intuitive difference between these cases: informally speaking, while the Sierpi\'nski triangle is a fairly ``thick'' fractal, self-similar sets satisfying the strong separation condition are quite loose. It raises the natural question whether there are fractals adhering to an intermediate behaviour in the following sense:  for small values of
$\aaa$ even the level sets 
of H\"older-$\aaa$ functions are sufficiently flexible and ``compressible''
and there exists $1>\aaa_{\fff}>0$  such that $D_{*}(\alpha, F)=D_{*}(0, F)$, holds for all $\aaa\in [0,\aaa_{\fff})$ while  $D_{*}(\alpha, F)>D_{*}(0, \Delta)$
holds for $\aaa >\aaa_{\fff}$. 
If this happens we say that there is a {\it phase transition} for $D_{*}(\alpha, F)$. 
In a very rough heuristic way we could say that if there is a phase transition then
for small values of $\aaa$ the ``traffic'' corresponding to the level sets is not heavy enough to generate ``traffic jams'' and can go through the ``narrowest'' places,
while for larger $\aaa$s ``traffic jams'' show up and ``thicker'' parts of the fractal
should be used to ``accommodate'' the level sets.

Next we construct a fractal $F$ for which $D_{*}(\alpha, F)=D_{*}(0, F)=0$ for some small values of $\alpha$, while $D_{*}(\alpha, F)>0$ for large values of $\alpha$. Corollary \ref{cor:biLipIFS} hints us that we can hope for simple examples displaying this phenomenon, however, probably not self-similar ones.

To this end, we construct a fat Cantor set $C= \bigcap_{n=0}^{\infty}C_n$, where $(C_n)$ is a decreasing sequence of sets, such that $C_n$ is the union of $2^n$ disjoint, closed intervals  of  the same length.  
Let $C_0=[0,1]$, and for $n>0$ we obtain $C_{n}$ by removing an open interval from the middle of each maximal subinterval of $C_{n-1}$. 
We make the construction explicit by specifying the length of the maximal subintervals at each level: let it be $l_n=\frac{1}{2^{n+1}-1}$. 
Then $2l_n<l_{n-1}$, hence such a system can be constructed indeed by successive interval removals. 
Moreover, the Cantor set $C$ in the limit is indeed a fat Cantor set in terms of Lebesgue measure, as $\lambda(C_n)=\frac{2^n}{2^{n+1}-1}$, thus
\begin{displaymath}
\lambda(C)=\lim_{n\to\infty}\lambda(C_n)=\frac{1}{2}.
\end{displaymath}

We can verify the following:


\begin{theorem} \label{phasetransfatcantor}
$F=C\times C \subseteq \mathbb{R}^2$ admits phase transition. 
Notably, for $0<\alpha<\frac{1}{2}$ we have $D_{*}(\alpha, F)=0$, while for $\frac{1}{2}<\alpha\leq 1$ we have $D_{*}(\alpha, F)= 1$.
\end{theorem}

\begin{figure}[ht]
\includegraphics[width=0.55\textwidth]{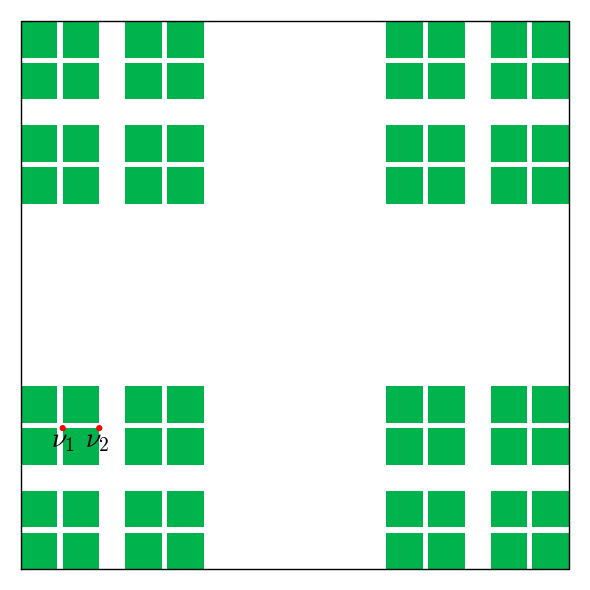}
\caption{Step $3$ of construction of $F$\label{figphase}}
\end{figure}

While the statement concerning small exponents will easily follow from Theorem \ref{sscdegenerate}, the other part is more technical. It requires a lemma, for which we will need the notion of Hausdorff capacity:

\begin{definition}
The $\alpha$ dimensional Hausdorff capacity of a set $E\subseteq \mathbb{R}^p$ is
$$\Lambda^{\alpha}(E)=\inf\left\{\sum_{i=1}^{\infty}|U_i|^\alpha: \text{ }E\subseteq \bigcup_{i=1}^{\infty}U_i\quad \text{ for some }(U_i)_{i=1}^{\infty}\right\}.$$
\end{definition}

The Hausdorff capacity is closely related to the problem we consider: it gives an upper estimate for the measure $\lambda(f(E))$ if $f$ is a 1-H\"older-$\alpha$ function.


\begin{lemma} \label{lemma:capacitylimit}
Let $I_{k}$ be a maximal subinterval of $C_k$ and $\frac{1}{2}<\alpha\leq 1$. Then
$$\frac{\Lambda^{\alpha}(I_{k}\setminus C)}{|I_{k}|}\to 0,$$
as $k\to \infty$.
\end{lemma}

\begin{proof}
Let $k\geq{1}$. 
By construction, $|I_{k}|=\frac{1}{2^{k+1}-1}>\frac{1}{2^{k+1}}$. 
We also know that the length $r_m$ of an interval removed from $C_{m-1}$ to obtain $C_{m}$ can be estimated from above by
\begin{equation}\label{removal_above_estimate}
r_m = l_{m-1} - 2l_m = \frac{1}{2^m-1}-\frac{2}{2^{m+1}-1}=\frac{1}{(2^m -1)(2^{m+1}-1)}<\frac{1}{2^{2m}}\end{equation}
for $m>2$. 
Now cover the set $I_{k}\setminus C$ by intervals contiguous to $C$ in $I_{k}$. 
It is easy to see that this covering consists of intervals of length $r_m$ for some $m>k$, and the number of intervals with length $r_m$ is $2^{m-k-1}$. 
Consequently, 
\begin{equation}
\Lambda^{\alpha}(I_{k}\setminus C)\leq \sum_{m=k+1}^{\infty}2^{m-k-1}r_m^{\alpha}\leq \sum_{m=k+1}^{\infty} 2^{m-k-1-2m\alpha},
\end{equation}
where we use (\ref{removal_above_estimate}) for the second estimate. 
The geometric series is summable for $\frac{1}{2}<\alpha\leq 1$, and it yields
\begin{equation}
\Lambda^{\alpha}(I_{k}\setminus C)\leq 2^{-k-1} \frac{2^{(k+1)(1-2\alpha)}}{1-2^{1-2\alpha}}=\frac{2^{-2\alpha(k+1)}}{1-2^{1-2\alpha}}.
\end{equation}
Consequently, for $\frac{1}{2}<\alpha\leq 1$
\begin{equation}
 \frac{\Lambda^{\alpha}(I_k\setminus C)}{|I_k|} \leq \frac{ 2^{(k+1)(1-2\alpha)}}{1-2^{1-2\alpha}}\to 0, \text{ as $k\to \infty$, } 
\end{equation}
which concludes the proof. \end{proof}

\begin{proof}[Proof of Theorem \ref{phasetransfatcantor}]
The first statement about $0<\alpha<\frac{1}{2}$ simply follows from Theorem \ref{sscdegenerate}, as $F$ has a $\left(\frac{1}{2}, \frac{1}{4}\right)$ separated structure. This observation follows easily from calculations carried out in the proof of 
Lemma \ref{lemma:capacitylimit}: notably, if $\mathcal{S}_k$ consists of the sets of the form $F\cap (I_j\times I_{j'})$, where $I_j$ and $I_{j'}$ are (not necessarily different) maximal subintervals of $C_k$ then each element of $\mathcal{S}_k$ has diameter 
$$\sqrt{2}\cdot \frac{1}{2^{k+1}-1}\leq \sqrt{2} \cdot 2^{-k}.$$ Moreover, for $k\geq{2}$ one can easily deduce that  the  distance between different elements of $\mathcal{S}_k$ is at least 
$$\frac{1}{(2^k -1)(2^{k+1}-1)}\geq \frac{1}{4} 4^{-k},$$
using the first part of \eqref{removal_above_estimate} and the fact that as the elements of $\mathcal{S}_k$ are product sets, they differ in one of their factors. It verifies that $F$ has a $\left(\frac{1}{2}, \frac{1}{4}\right)$ separated structure, and yields the first part of the theorem due to Theorem \ref{sscdegenerate} and
$$\frac{\log\frac{1}{2}}{\log\frac{1}{4}}=\frac{1}{2}.$$

For the second statement, 
by Theorem \ref{thm:trivial_upper_bound} we have $D_{*}(\alpha, F)\leq 2- 1=1$
and hence it is sufficient to show that $D_{*}(\alpha, F)\geq 1 $ holds for 
$\frac{1}{2}<\aaa\leq 1$.

Recall that the union of all the $c$-H\"older-$\alpha$ functions for $0<c<1$ defined on $F$ is a dense subset of 1-H\"older-$\alpha$ functions in the supremum norm. 
Consequently it would be sufficient to verify that for a fixed $c$-H\"older-$\alpha$ function $f$ and $\varepsilon>0$ we can find a 1-H\"older-$\alpha$ function $\tilde{f}\in B(f,\varepsilon)$ and $\varepsilon'>0$ such that for any 1-H\"older-$\alpha$ function $g\in B(\tilde{f},\varepsilon')$  we have $\dim_H(g^{-1}(r))\geq{1}$ in a set of positive measure  of $r$s.  
In fact, it would verify that $D_{*}^{g}(F)=1$ on a dense open set, which clearly yields that it is the generic behaviour.

As $f$ is $c$-H\"older-$\alpha$, $f+h$ is a 1-H\"older-$\alpha$ function if $h$ is $(1-c)$-H\"older-$\alpha$. 
We will use this property to introduce the perturbed function $\tilde{f}$, for which some $k$th level cylinder  of $C\times C$  (which is a square) has adjacent vertices $v_1=(x_1,y_1), v_2=(x_2,y_1)$ such that
\begin{equation} \label{large_change}
|\tilde{f}(v_1)-\tilde{f}(v_2)|\geq (1-c)|v_1-v_2|=(1-c)|x_1-x_2|.
\end{equation}
More explicitly, choose $k$ large enough such that for a maximal subinterval $I=[x_1,x_2]\subseteq C_k$ we have
\begin{equation} \label{capacitybound}
\frac{\Lambda^{\alpha}(I\setminus C)}{|x_1-x_2|}<\delta,
\end{equation}
where $\delta$ is to be fixed later. 
By Lemma \ref{lemma:capacitylimit}, this estimate holds for large enough $k$. 
We can assume without loss of generality that $f(v_1)\leq f(v_2)$ for the vertices $v_1=(x_1,y_1), v_2=(x_2,y_1)$ of some $k$th level cylinder  of $C\times C$ , as the other case is similar. 
We can also assume that these vertices are top vertices of that $k$th level cylinder
 see Figure \ref{figphase}. 
Hence if we define
\begin{equation} h(x,y) = 
	\begin{cases*}
      0, & if $x<x_1$, \\
      (1-c)(x-x_1), & if $x_1\leq x \leq x_2$, \\
      (1-c)(x_2-x_1) & otherwise,
    \end{cases*}
\end{equation}
then $\tilde{f}=f+h$ satisfies (\ref{large_change}).

We take a $\delta'>0$ which will be specified later.
By continuity, we can choose $r$ such that for any $y\in [y_1 -r, y_1]\cap C$ we have 
$$|\tilde{f}(x_1,y_1)-\tilde{f}(x_1,y)|<\delta' \text{ and } |\tilde{f}(x_2,y_1)-\tilde{f}(x_2,y)|<\delta'.$$
Consequently, if $g \in B(\tilde{f}, \delta')$, then 
$$|g(x_1,y_1)-g(x_1,y)|<3\delta'
\text{ and }
|g(x_2,y_1)-g(x_2,y)|<3\delta'.$$
Besides that, as $(x_1,y_1), (x_2,y_1)$ were chosen as top vertices of cylinders  of $C\times C$,  
$$\lambda(C\cap[y_1 - r, y_1])=:\eta>0.$$
Now by Theorem 1 of \cite{[GrunbHolderext]} we can extend $g$ to a 1-H\"older-$\alpha$ function defined on $[0,1]^2$. 
Denote the extended function by $g$ as well. 
Due to the choice of $r$, the continuity of the extended function, and the intermediate value theorem, we have that the $g$-image of the planar line segment $[(x_1,y),(x_2,y)]$  for any $y\in [y_1 -r, y_1]\cap C$ contains the interval $[\tilde{f}(v_1)+3\delta', \tilde{f}(v_2)-3\delta']$. 
This interval has measure at least $(1-c)(x_2-x_1)-6\delta'$. 
Moreover, as $[(x_1,y),(x_2,y)]\setminus F \subseteq [(x_1,y),(x_2,y)]$ is congruent to $I\setminus C \subseteq I$, due to (\ref{capacitybound}) and the fact that $g$ is 1-H\"older-$\alpha$, we have
$$\lambda(g([(x_1,y),(x_2,y)]\setminus F))\leq \delta(x_2-x_1).$$
Consequently, the remainder measure of values is taken on $[(x_1,y),(x_2,y)]\cap F$, yielding that  $g([(x_1,y),(x_2,y)]\cap F)\cap [\tilde{f}(v_1)+3\delta', \tilde{f}(v_2)-3\delta']$  has measure at least $(1-c-\delta)(x_2-x_1)-6\delta'$. 
Fix now the values of $\delta$ and $\delta'$ such that this quantity is positive. 

By the above calculations, we can conclude that we have
\begin{equation} \label{eq:Fubini}
\begin{split} \lambda_2\left\{(g(x,y),y): g(x,y)\in [\tilde{f}(v_1)+3\delta', \tilde{f}(v_2)-3\delta'], \text{ } y\in [y_1-r,  y_1]\cap C, \text{ } x\in [x_1, x_2]\cap C\right\}& \\ \geq \eta\cdot((1-c-\delta)(x_2-x_1)-6\delta')>0, \end{split}
\end{equation}
where $\lambda_2$ denotes the two-dimensional Lebesgue measure. Note that this set is measurable indeed as it is the image of the compact set
$$([x_1, x_2]\cap C) \times ([y_1-r,  y_1]\cap C),$$
under the continuous mapping $(x,y)\to (g(x,y), y)$.
However, by Fubini's theorem, we can rewrite the measure in \eqref{eq:Fubini} as
\begin{equation}
\int_{t=\tilde{f}(v_1)+4\delta'}^{\tilde{f}(v_2)-4\delta'}\lambda\left\{y: y\in [y_1-r,  y_1]\cap C \text{ and } g(x,y)=t \text{ for an } x\in[x_1,x_2]\cap C\right\}dt.
\end{equation}
As this integral is positive, the integrand is positive on a set $A$ of positive measure. 
That is, for any $t\in A$ we have that 
$$\lambda\left\{y: y\in [y_1-r,  y_1]\cap C \text{ and } g(x,y)=t \text{ for an } x\in[x_1,x_2]\cap C\right\}>0.$$
which is equivalent to that the projection of
$$g^{-1}(t) \cap ([x_1,x_2]\times[y_1-r,  y_1])\cap F$$
to the second coordinate has positive measure. 
That is, the projection has Hausdorff dimension 1, which obviously yields that $g^{-1}(t) \cap ([x_1,x_2]\times[y_1-r,  y_1]) \cap F$ has Hausdorff dimension at least 1 as well for a set of $t$s with positive measure. 
It concludes the proof. \end{proof}

\section{Conclusions}\label{*secconc}

Considering Hausdorff dimensions of level sets of generic $1$-H\"older-$\aaa$ functions we further refined our concepts which lead to the definition of topological Hausdorff dimension
introduced in \cite{BBEtoph}.
This paper follows \cite{sier} which is  a more theoretical paper. 

In the current paper by explicit examples and calculations
we illustrated why these concepts are related to  "thickness/narrow cross-sections'' of a ‘‘network’’ corresponding to a fractal set.

We gave detailed calculations for the Sierpi\'nski triangle, $\DDD$.
To obtain a lower estimate in this case for the Hausdorff dimension of almost every level-set  of {\it any 1-H\"older-$\alpha$ function } defined on $\DDD$
we used a concept of conductivity of some subtriangles. 

For the Hausdorff dimension 
of almost every level-set of  {\it a generic 1-H\"older-$\aaa$ function } defined on $\DDD$ we also calculated an upper estimate. 

Apart from the Sierpi\'nski triangle, fractals with strong separation condition,
 and ones with a certain separated structure were also considered. An example based on these latter ones illustrates
a phenomenon which can be regarded as phase transition in our estimates.  
 This  means that for smaller values of the Hölder exponent $\alpha$ level sets of 
generic 1-H\"older-$\alpha$ functions are as flexible/compressible as those
of a continuous function. 
 While for larger values of $\alpha$ the geometry of the fractal also matters.
 In a very rough heuristic way one could say that if there is a phase transition then
for small values of $\aaa$ the ``traffic'' corresponding to the level sets is not heavy enough to generate ``traffic jams'' and can go through the ``narrowest'' places,
while for larger $\aaa$s ``traffic jams'' show up and ``thicker'' parts of the fractal
should be used to ``accommodate'' the level sets.

This work can also initiate similar calculations and estimations on many other fractals.





\bibliographystyle{amsplain} 
\bibliography{sier}

\end{document}